\documentclass[a4paper,final,12pt]{article}
\usepackage{epstopdf}
\usepackage[colorinlistoftodos]{todonotes}
\usepackage{epsfig}
\usepackage{graphics}
\usepackage{amsmath}
\usepackage{amsthm}
\usepackage{amssymb}
\usepackage{bm}
\usepackage{ulem}
\usepackage{geometry}
\newtheorem{theorem}{Theorem}[section]

\newtheorem{lemma}[theorem]{Lemma}
\newtheorem{proposition}[theorem]{Proposition}

\newcommand{\SF}[1]{\textcolor{black}{{#1}}}
\newcommand\redout{\bgroup\markoverwith{\textcolor{red}{\rule[.5ex]{2pt}{0.8pt}}}\ULon}
\title{Variational multi-scale spectral solution of convection-dominated parabolic problems}
\author{T. Chac\'on Rebollo\footnotemark[1]
\and S. Fern\'andez-Garc\'ia\footnotemark[1]}

\begin{document}

\footnotetext[1]
{Dpto. EDAN \& IMUS, University of Seville, Campus de Reina Mercedes, 41012 Sevilla (Spain), e-mail: chacon@us.es, soledad@us.es}
\maketitle
\begin{abstract}
In this work, we consider an extension to parabolic problems of the variational multi-scale method with spectral approximation of the sub-scales. 
We first discretize in time using a finite difference scheme and second, apply the generalization of the spectral variational multi-scale method.
To obtain error estimations in convection-dominated flows, we find a helpful link between the stabilized term expressed in terms of Green's functions and in terms of spectral functions.
Finally, we present some numerical tests to show the reliability of the method. We consider the stationary one-dimensional advection-diffusion-reaction equation and the evolutive one-dimensional advection-diffusion equation.  
\end{abstract}

\section{Introduction}
The Variational Multi-Scale method provides a general framework to remedy the stability difficulties associated to the Galerkin discretization of PDEs (partial differential equations) with terms of different derivation orders (see Hughes (cf. \cite{Hughes_1995,HughesStewart_1995,HughesFMQ_1998}). At the discrete level, spurious oscillations may appear when certain low-order operator terms are dominant, providing unreliable numerical solutions for technological and engineering applications.  The basic stabilized method in the framework of finite element discretizations is the SUPG (Streamline Upwind Petrov-Galerkin) method, (see \cite{Brooks_1982}). It consists in adding to the classical Galerkin formulation an extra term devoted to control the advection derivative. This pioneering work was followed by a large class of stabilized methods (Galerkin-Least Squares methods, adjoint (or unusual) Galerkin-Least Squares methods, among others) all consisting in adding extra terms to the Galerkin formulation aiming to control one or several operator terms appearing in the equations. These methods where mainly applied to the numerical solution of incompressible and subsequently compressible flow equations, also proving that they provide a further stabilization of the discretization of the pressure gradient. An overview of those methods may be found in \cite{Hughes_1995}.
\par
 The Variational Multi-Scale (VMS) formulation states separate variational problems for large and small scales. The small scales are driven by the residual of the large scales. A global stabilization effect is achieved, due to a dissipative effect of the small scales onto the large scales. To build a feasible VMS method, the small scale problem is further discretized by some kind of approximation. A possibility is a local diagonalization of the PDE operator on each grid element. This leads to the Adjoint stabilized method, mentioned above, and also to the Orthogonal Sub-Scales (OSS) method, introduced by Codina in \cite{codina12}.  In these methods, the effects of the sub-grid scales onto the resolved ones is made apparent through a dissipative interaction of operator terms. The VMS methods have been successfully applied to many flow problems, and in particular to the building of  models of Large Eddy Simulation (LES) of turbulent flows, with remarkable accuracy (cf. \cite{hug112,john0612,Chaconlibro}). 
 \par The VMS method has been successfully applied to evolution PDEs. Early studies date back to the 1990s, when results from \cite{Hughes_1995} were extended to nonsymmetric, linear, evolution operators, see \cite{HughesStewart_1995}. After that, we find different works in the literature dealing with this class of problems, such us those commented subsequently.  In \cite{harari,hauke}, 
the authors consider parabolic problems where spurious oscillations occur in the Galerkin formulation due to extra small time steps. In order to remove this pathology, they used the Rothe Method, also called the Horizontal Method of Lines, 
which consists in first performing a semi-discretization in time, and then applying a stabilized method to the spatial problems issued from the time discretization. 
In the series of articles \cite{fourier1,fourier2,fourier3} they consider the transient Galerkin and SUPG methods, the transient subgrid scale (SGS) stabilized methods and the transient subgrid scale/gradient subgrid scale (SGS/GSGS), respectively, 
and perform Fourier analysis for the one-dimensional advection-diffusion-reaction equation.
On the other hand, a stabilized finite element method to solve the transient Navier-Stokes equations based on the decomposition of the unknowns into resolvable and subgrid scales is considered in \cite{codina2002,codina2007}.
Finally, in \cite{asensio} the authors consider the evolutive advection-diffusion-reaction problem in one space dimension and compare the Rothe method with the so-called Method of Lines, which consists on first, discretize in space by means of a stabilized finite element method and then use a finite difference scheme to approximate the solution. 
 
 
 \par
The use of spectral techniques to model the sub-grid scales is introduced in \cite{ChaconDia}. The sub-grid scales are initially approximated by bubble functions on each grid element. The basic observation is that the eigenpairs of the advection-diffusion operator may be calculated explicitly, what allows to analytically calculate the sub-grid scales by means of a spectral expansion on each grid element. A feasible VMS-spectral discretization is then built by truncation of this spectral expansion to a finite number of modes. For piecewise affine finite elements, the stabilization coefficients are weighted sums of the cha\-rac\-te\-ris\-tic times of the eigenmodes. The method with an odd number of modes satisfies the discrete maximum principle. It is found to be of 3rd. order with respect to the number of eigenmodes.
\par
In the present paper we apply the method of \cite{ChaconDia} to the solution of evolution advection-diffusion equations. We follow the Rothe Method (Horizontal Method of Lines \cite{asensio,harari,hauke}) applying the spectral VMS discretization to the advection-diffusion-reaction problems issued from the time discretization. For these problems we cast the method as a standard VMS method with stabilized coefficients replaced by some approximated stabilized coefficients. These are computed from either the spectral eigenfunctions or from approximated element Green's functions, that in their turn are exactly computed from these eigenfunctions.
We obtain error estimates for both diffusion-dominated and convection-dominated regimes. For the latter we prove that the method is accurate for a range of local P\'eclet numbers that increases as the number of eigenfunctions appearing in the method increases.
\par
We present several numerical tests for 1D evolution advection-diffusion equations. We observe that the method still satisfies the maximum principle for evolution advection-diffusion equations when the number of eigenfunctions is odd. Also, that the numerical solution presents a super-convergence effect at grid nodes: it is second order accurate in discrete $L^2(0,T;H^1(\Omega))$ norm at these nodes, while the first time iterate is exact at grid nodes. We further show that the method presents a fourth order convergence with respect to the number of eigenfunctions.
\par
The article is outlined as follows. Section \ref{formulation} is devoted to the formulation of the problem in terms of the spectral approximation of the sub-scales. In Section \ref{evolutive} we build the method for the evolutive advection-diffusion problem. After that, in Section \ref{secestimates}, we include the error estimates analysis, where we distinguish between the diffusion-dominated regime and the convection-dominated regime. In Section \ref{numerics} we perform some numerical tests and finally, in Section \ref{conclusions} we present some conclusions and open problems to be addressed. We also include Appendix \ref{secapp} to expose two technical results.
\section{Abstract formulation and spectral approximation of the sub-scales}\label{formulation}
We consider a Hilbert space $H$. We identify $H$ with its topological dual $H'$. We consider another Hilbert space V. We assume $V \subset H$  with dense and bounded embedding so that $H' \subset V'$. Denote by ${\cal L}_2(V)$ the space of bilinear bounded forms on $V$. Let $a \in L^1(0,T; {\cal L}_2(V))$ uniformly bounded and $V$-elliptic with respect to $t$. Let $f \in L^2(0,T;X')$ and $U_0 \in H$. Consider the variational parabolic problem,
\begin{equation}\label{AWEAD}
\left\{\begin{array}{l}
\mbox{Find }U\in L^2((0,T);X)\cap C^0([0,T];H) \quad \mbox{such that,}\\  \noalign{\smallskip}
\displaystyle \frac{d}{dt}(U,V)+a(t;U,V)=\langle f,V\rangle \,\, \forall V\in X,\,\, \mbox{in  } {\cal D}'(0,T);\\  \noalign{\smallskip}
U(0)=U_0.
\end{array}\right.
\end{equation}
It is standard that this problem admits a unique solution.
To discretize this problem, we proceed through the so-called Horizontal Method of Lines \cite{asensio,harari}, which consists on first, discretize in time using a finite difference scheme and second apply the spectral Variational Multi-Scale method. 

Consider a uniform partition $\{0=t_0<t_1<...<t_N=T\}$ of the interval $[0,T],$ with time-step size $k=T/N.$ Then, the time discretization of problem (\ref{AWEAD}) by the Backward Euler scheme gives
\begin{equation}\label{ecu_abstractdis}
\begin{array}{l}
(U^{n+1},V)+k a^{n+1}(U^{n+1},V)=k\langle f^{n+1},V\rangle+(U^{n},V), \quad \forall V\in X,\\ \noalign{\smallskip}
n=0,1,...,N-1,U^0=U(0),
\end{array}
\end{equation}
which can be seen as a family of stationary problems. The data $a^{n+1}$ and $f^{n+1}$ are some approximate values to $a(t;\cdot,\cdot)$ and $f(t)$ at $t=t_{n+1}$. Thus, we can define,
\begin{equation*}
\left\{\begin{array}{l}
B_n(U,V)=(U,V)+k a^{n+1}(U,V),\quad \forall U,V\in X,\\ \noalign{\smallskip}
l^{n+1}(V)=k \langle f^{n+1},V\rangle+(U^{n},V), \quad \forall V\in X,
\end{array}\right.
\end{equation*}
and rewrite problem (\ref{ecu_abstractdis}) as
\begin{equation*}\label{ecu_abstractDia}
\begin{array}{l}
B_n(U^{n+1},V)=l^{n+1}(V), \quad \forall V\in X,\\ \noalign{\smallskip}
n=0,1,...,N-1.
\end{array}
\end{equation*}
Now, it is possible to build the Variational Multi-Scale formulation of this problem. Indeed, we consider the decomposition,
\begin{equation}\label{dec}
X=X_h\oplus \tilde{X},
\end{equation}
where $X_h$ is a sub-space of $X$ of finite dimension, and $\tilde{X}$ is a complementary, infinite-dimensional, sub-space of $X.$ Notice that this is a multi-scale decomposition of the space $X,$ being $X_h$ the large scale space and $\tilde{X}$ the small scale space. Hence, one can decompose the solution of problem (\ref{ecu_abstractDia}) as 
\begin{equation*}\label{Udec}
\begin{array}{l}
U^{n+1}=U_h^{n+1}+\tilde{U}^{n+1}, \quad \mbox{for } U_h^{n+1}\in X_h, \tilde{U}^{n+1}\in \tilde{X}, \\ \noalign{\smallskip}
n=0,1,...,N-1,
\end{array}
\end{equation*}
and in the same form the test function $V=V_h+\tilde{V}.$ Therefore, problem (\ref{ecu_abstractDia}) can be reformulated as
\begin{equation}\label{ecu_dec2}
\left\{\begin{array}{lr}
B_n(U_h^{n+1}+\tilde{U}^{n+1},V_h)=l^{n+1}(V_h),\quad\forall V_h\in X_h,&(a)\\ \noalign{\smallskip}
B_n(U_h^{n+1}+\tilde{U}^{n+1},\tilde{V})=l^{n+1}(\tilde{V}),\quad\forall  \tilde{V}\in \tilde{X},&(b)\\ \noalign{\smallskip}
n=0,1,...,N-1.
\end{array}\right.
\end{equation}
From equation (\ref{ecu_dec2})(b), we can define the residual of the large scales component in each temporal step as
\begin{equation}\label{residuo}
\begin{array}{l}
\langle R_n(U_h^{n+1}),\tilde{V}\rangle=l^{n+1}(\tilde{V})-B_n(U_h^{n+1},\tilde{V}),\quad\forall  \tilde{V}\in \tilde{X}, \\
n=0,1,...,N-1.
\end{array}
\end{equation}
Consider the static condensation operator $\Pi_n:\tilde{X}'\mapsto\tilde{X}$ defined by $\Pi_n(g)=\tilde{G}\in\tilde{X},$ solution of
$$
B_n(\tilde{G},\tilde{V})=\langle g, \tilde{V} \rangle,\,\, \forall \tilde{V} \in \tilde{X}.
$$
Then $\tilde{U}^{n+1}=\Pi_n(R_n(U_h^{n+1}))$ with $U_h^{n+1}$ the solution of 
\begin{equation}\label{ecuVMS}
\begin{array}{l}
B_n(U_h^{n+1},V_h)+B_n(\Pi_n(R_n(U_h^{n+1})),V_h)=l^{n+1}(V_h), \quad \forall V_h\in X_h,\\ \noalign{\smallskip}
n=0,1,...,N-1,
\end{array}
\end{equation}
which is the standard Variational Multi-Scale (VMS) reformulation of problem (\ref{ecu_abstractDia}).

It is possible to do an spectral approximation of the small scales with our modified operator.
Assume that (\ref{ecu_abstractdis}) is the variational formulation of the PDEs
\begin{equation}\label{ecu_abstractst}
\begin{array}{l}
\tilde{\mathcal{L}_n}(U^{n+1})=l^{n+1},\\ \noalign{\smallskip}
n=0,1,...,N-1,U^0=U(0),
\end{array}
\end{equation}
on a bounded domain $\Omega\subset\mathbb{R}^d,$ with $\tilde{\mathcal{L}_n}=I+k \mathcal{L}_n$ , being $I$ the identity operator, $\mathcal{L}_n$ the elliptic operator defined by
\begin{equation*}\label{operator}
\begin{array}{l}
\langle{\mathcal{L}_n}V,W\rangle=a^{n+1}(V,W),\quad \forall W\in X,\\ \noalign{\smallskip}
\end{array}
\end{equation*}
  and $X$ is a suitable Hilbert space of functions defined on $\Omega$.
Then, $\tilde{\mathcal{L}_n}:X\mapsto\tilde{X}$ is the operator defined by
\begin{equation*}\label{ecu_abstractst}
\begin{array}{l}
\langle\tilde{\mathcal{L}_n}V,W\rangle=B_n(V,W)=(V,W)+ka^{n+1}(V,W),\quad \forall W\in X.\\ \noalign{\smallskip}
\end{array}
\end{equation*}
Given a triangulation $\mathcal{T}_h$ of the domain $\Omega,$ we can approximate the small scale space $\tilde{X}$ by 
$$\tilde{X}=\bigoplus_{K\in\mathcal{T}_h}\tilde{X}_K,\quad\mbox{with } \tilde{X}_K=\{\tilde{V}\in\tilde{X}:supp(\tilde{V})\subset K\}.$$
Hence, it is possible to approximate
\begin{equation*}
\tilde{U}\simeq \sum_{K\in\mathcal{T}_h}\tilde{U}_K,\quad\mbox{with }\tilde{U}_K\in\tilde{X}_K,
\end{equation*}
and the problem (\ref{residuo}) is approximated by the family of problems
\begin{equation*}\label{residuo_K}
\begin{array}{l}
\langle R_n(U_h^{n+1}),\tilde{V}_K\rangle=B_n(\tilde{U}_K^{n+1},\tilde{V}_K),\quad\forall\tilde{V}_K\in \tilde{X}_K, K\in\mathcal{T}_h, \\ \noalign{\smallskip}
n=0,1,...,N-1.
\end{array}
\end{equation*}
Then, $\tilde{U}_K^{n+1}=\Pi_{n,K}(R_n(U_h^{n+1})),$ where $\Pi_{n,K}$ denotes the restriction of operator $\Pi_n$ to  $\tilde{X}_K.$

Given a weight function $p$ on $K$ (a measurable real function which is positive a.e. on $K$) let us define the weighted  $L^2$ space
\begin{equation*}
L^2_p(K)=\{W:K\rightarrow \mathbb{R} \mbox{ measurable such that } p|W|^2\in L^1(K)\},
\end{equation*}
which is a Hilbert space endowed with the inner product 
\begin{equation*}
(W_1,W_2)_p=\int_K p(x) W_1(x) W_2(x) dx.
\end{equation*}
Thus, the next result, whose proof is analogous to the proof of Theorem 1 in \cite{ChaconDia}, is satisfied.
\begin{theorem}\label{maintheorem}
Let us assume that there exists a complete sub-set $\{\hat{z}_j^{(n,K)}\}_{j\in\mathbb{N}}$ on $\tilde{X}_K$ formed by eigenfunctions of the operator $\mathcal{L}_{n,K}$, which is an orthonormal system in $L^2_{p_{n,K}}(K)$ for some weight function $p_{n,K}\in C^1(\bar{K}).$ Then,
\begin{equation}\label{series}
\begin{array}{l}
\tilde{U}_K^{n+1}=\displaystyle\sum_{j=0}^{\infty}\beta_j^{K,n} \langle R_n(U_h^{n+1}),p_{n,K} \hat{z}_j^{(n,K)}\rangle \hat{z}_j^{(n,K)},\quad \mbox{with }\beta_j^{K,n} =(\Lambda_j^{K,n})^{-1},\\
n=0,1,...,N-1,
\end{array}
\end{equation}
where $\Lambda_j^{K,n} =1+k\lambda_j^{(n,K)},$ being $\lambda_j^{(n,K)}$ the eigenvalue of $\mathcal{L}_{n,K}$ associated to $\hat{z}_j^{(n,K)}.$
\end{theorem}
Note that now the coefficients in series (\ref{series}) depends on the time step, in contrast to \cite{ChaconDia}, where the considered problem was stationary.

Finally, in order to obtain a feasible discretization, we truncate series (\ref{series}) to $M\geq 1$ addends and approximate problem (\ref{ecuVMS}) by
\begin{equation}\label{ecuVMS_M}
\begin{array}{l}
B_n(U_{h,M}^{n+1},V_h)+B_n(\Pi^M_n(R_n(U_{h,M}^{n+1})),V_h)=l^{n+1}(V_h), \quad \forall V_h\in X_h,\\ \noalign{\smallskip}
n=0,1,...,N-1,
\end{array}
\end{equation}
where the unknowns are $U_{h,M}^{n+1}\in X_h$ and the operator $\Pi^M_n$ is given by
\begin{equation}\label{Pi_M}
\Pi^M_n(\varphi)=\sum_{K\in\mathcal{T}_h}\Pi^M_{n,K}(\varphi), \,\mbox{with }\Pi^M_{n,K}(\varphi)=\sum_{j=0}^{M}\beta_j^{K,n} \langle \varphi,p_{n,K} \hat{z}_j^{(n,K)}\rangle \hat{z}_j^{(n,K)}.
\end{equation}

\section{Application to evolutive advection-diffusion problem}\label{evolutive}
In this section, we apply the spectral method (\ref{ecuVMS_M}) to the following initial value problem with homogeneous boundary conditions.
Let us consider the evolutive advection-diffusion problem
\begin{equation*}\label{EAD}
\left\{\begin{array}{l}
\displaystyle\frac{\partial}{\partial t}U+\mathbf{c}\cdot\nabla U-\mu \Delta U=f \quad \mbox{in }\Omega\times(0,T),\\  \noalign{\smallskip}
U=0 \quad \mbox{on }\partial \Omega\times(0,T),\\  \noalign{\smallskip}
U(0)=U_0\quad \mbox{on } \Omega,\\
\end{array}\right.
\end{equation*}
where $\Omega\subset\mathbb{R}^d$ $d\geq 1$ is a bounded domain, $\mathbf{c} \in L^\infty((0,T)\times \Omega)^d$ is a divergence-free given velocity field, $\mu>0$ is the diffusion coefficient,  $f\in L^2((0,T);L^2(\Omega))$ is the source term and $U_0\in L^2(\Omega)$ is the initial data.
The weak formulation of problem (\ref{EAD}) is given by
\begin{equation}\label{WEAD}
\left\{\begin{array}{l}
\mbox{Find }U\in L^2((0,T);H^1_0(\Omega))\cap C^0([0,T];L^2(\Omega)) \quad \mbox{such that,}\\  \noalign{\smallskip}
(\partial_t U,V)+(\mathbf{c}\cdot\nabla U,V)+\mu (\nabla U,\nabla V)=\langle f,V\rangle \quad \forall V\in H_0^1(\Omega),\\  \noalign{\smallskip}
U(0)=U_0,
\end{array}\right.
\end{equation}
which is problem (\ref{AWEAD}) with 
\begin{equation}\label{def_a}
a(t,U,V)=(\mathbf{c}(t)\cdot\nabla U,V)+\mu (\nabla U,\nabla V).
\end{equation}
Therefore, it is possible to consider the spectral VMS discretization (\ref{ecuVMS_M}).

Note that equations in expression (\ref{ecu_abstractdis}) for the bilinear form considered in this case (\ref{def_a}), can be seen as a family of stationary advection-diffusion-reaction problems. In particular, if instead of the coefficient 
$1/k$ coming from the semi-discretization in time, we consider a general reaction coefficient $\gamma$, we have a general stationary advection-diffusion-reaction 
problem,
\begin{equation}\label{EADR}
\left\{\begin{array}{l}
\gamma U+\mathbf{c}\cdot\nabla U-\mu \Delta U=g \quad \mbox{in }\Omega,\\  \noalign{\smallskip}
U=0 \quad \mbox{on }\partial \Omega,\\ 
\end{array}\right.
\end{equation}
where $\mathbf{c}\in L^\infty(\Omega)$ is a divergence-free given velocity field, $\gamma,\mu>0$ are the reaction and diffusion coefficients, respectively,   and $g\in L^2(\Omega)$ is the source term.
The weak formulation of problem (\ref{EADR}) is given by
\begin{equation}\label{WEADR}
\left\{\begin{array}{l}
\mbox{Find }U\in H_0^1(\Omega) \quad\mbox{such that,}\\  \noalign{\smallskip}
\gamma\,(U,V)+(\mathbf{c}\cdot\nabla U,V)+\mu \,(\nabla U,\nabla V)=(g,V) \quad \forall V\in H_0^1(\Omega).\\  \noalign{\smallskip}
\end{array}\right.
\end{equation}
Given a triangulation $\mathcal{T}_h$ of the domain $\Omega,$ we assume that the velocity $\mathbf{c}$  is approximated at time $t=t_n$ in the sub-grid term by a constant value $\mathbf{c}_{n,K}$ on each element K.
Let us now state a result about the eigenpairs of the advection-diffusion-reaction operator.
\begin{proposition}\label{propADR}
The couple $\left(\tilde{\omega}_j^{(n,K)},\eta_j^{(n,K)}\right)$ is an eigenpair of the advection-diffusion-reaction operator $\mathcal{L}_{n,K}$ if and only if the couple $\left(\tilde{W}_j^{(K)},\sigma_j^{(K)}\right)$ is an eigenpair of the Laplace operator $-\Delta$ in $H_0^1(K),$ where
\begin{equation}\label{sigma}
\tilde{\omega}_j^{(n,K)}=\psi^{(n,K)}\tilde{W}_j^{(K)}\quad\mbox{with}\quad \psi^{(n,K)}(\mathbf{x})=e^{\frac{1}{2\mu}(\mathbf{c}_n\cdot\mathbf{x})}\quad\mbox{and}\quad\eta_j^{(n,K)}=\gamma+\mu\left(\sigma_j^{(K)}+\frac{|\mathbf{c}_n|^2}{4\mu^2}\right),\forall j\in\mathbb{Z}.
\end{equation}
\end{proposition}
\begin{proof}
From Proposition 1 of \cite{ChaconDia}, we know that the couple $\left(\tilde{\omega}_j^{(n,K)},\lambda_j^{(n,K)}\right)$ is an eigenpair of the advection-diffusion operator if and only if the couple $\left(\tilde{W}_j^{(K)},\sigma_j^{(K)}\right)$ is an eigenpair of the Laplace operator $-\Delta$ in $H_0^1(K),$ where
$$\tilde{\omega}_j^{(n,K)}=\psi^{(n,K)}\tilde{W}_j^{(K)}\quad\mbox{with}\quad \psi^{(n,K)}(\mathbf{x})=e^{\frac{1}{2\mu}(\mathbf{c}_n\cdot\mathbf{x})}\quad\mbox{and}\quad\lambda_j^{(n,K)}=\mu\left(\sigma_j^{(K)}+\frac{|\mathbf{c}_n|^2}{4\mu^2}\right),\forall j\in\mathbb{N}.$$
Therefore, 
$$\begin{array}{l}
\gamma\left(\tilde{\omega}_j^{(n,K)},V\right)_K+\left(\mathbf{c}_n\cdot\nabla \tilde{\omega}_j^{(n,K)},V\right)_K+\mu \left(\nabla \tilde{\omega}_j^{(n,K)},\nabla V\right)_K\\ \noalign{\smallskip}
=\gamma\left(\tilde{\omega}_j^{(n,K)},V\right)_K+\lambda_j^{(n,K)}\left(\tilde{\omega}_j^{(n,K)},V\right)_K=
(\gamma+\lambda_j^{(n,K)})\left(\tilde{\omega}_j^{(n,K)},V\right)_K,
\end{array}$$
which concludes the proof.
\end{proof}
Note that, from expression (\ref{sigma}) when $\gamma=0,$ 
\begin{equation}
\eta_j^{(n,K)}=\mu\left(\sigma_j^{(K)}+\frac{|\mathbf{c}_n|^2}{4\mu^2}\right)=\lambda_j^{(n,K)}
\end{equation}
and we recover the eigenvalues of the advection-diffusion operator. Note also that the eigenfunctions do not depend on the reaction term, so they coincide to those of the advection-diffusion operator.
\SF{Moreover, the weight functions are given by
\begin{equation}\label{pesos}
p_{n,K}= (\psi^{(n,K)})^{-2}=e^{-\frac{1}{\mu}(\mathbf{c}_n\cdot\mathbf{x})},
\end{equation}
where $\psi^{(n,k)}$ is given in expression (\ref{sigma}).}

The eigenpairs of the Laplace operator can be exactly computed for elements with simple geometrical forms, as is the case of parallelepipeds. In the 1D case, these are
\begin{equation}\label{sigma2}
\tilde{W}_j^{(K)}=\sin\left(\sqrt{\sigma_j^{(K)}}(x-x_l)\right)\quad\mbox{with}\quad \sigma_j^{(K)}=\left(\frac{j\pi}{x_l-x_{l-1}}\right)^2 \quad\mbox{for} \quad l=1...,N, j\in\mathbb{Z}.
\end{equation}

Now that we have computed the eigenpair for the stationary advection-diffusion-reaction problem (\ref{EADR}), just by choosing $\gamma=1,$ $\mathbf{c}_n=k\mathbf{c}_n$ and $\mu=k\mu,$ we obtain the eigenpairs of the operator of equation (\ref{ecu_abstractdis}) with the bilinear form given in (\ref{def_a}).
Hence, from Theorem \ref{maintheorem} and bearing in mind Theorem 2 of \cite{ChaconDia}, one can explicitly compute the term $B_n(\Pi^M_n(R_n(U_{h,M}^{n+1})),V_h)$ (see Section \ref{numerics}). 
\section{Error Estimates}\label{secestimates}
In this section we estimate the error of the VMS spectral method in their application to the evolutive advection-diffusion problem. We distinguish two cases, the diffusion dominated regime
 and the advection dominated regime. In both cases, we will assume that $f$ and $\mathbf{c}_n$ are piecewise constant on each element $K,$ and will be denoted by $f_K$ and $\mathbf{c}_{n,K}.$ 

For the sake of brevity, we are going to use the following notation
\begin{equation*}\label{definf}
\begin{array}{l}
\displaystyle p_{\infty}=\max_{K\in\mathcal{T}_h, n=0,\cdots,N} \|p_{n,K}\|_\infty,\quad q_{\infty}=\max_{K\in\mathcal{T}_h, n=0,\cdots,N} \|p_{n,K}^{-1}\|_\infty \quad\mbox{and}\quad \|\mathbf{c}_n\|_\infty=\max_{K\in\mathcal{T}_h, n=0,\cdots,N}\|\mathbf{c}_{n,K}\|_\infty,\\ \noalign{\smallskip}
\end{array}
\end{equation*}
\SF{where $p_{n,K}$ is defined in expression (\ref{pesos}).}

The approximations $f^{n+1}$ to $f(t_{n+1})$ are assumed to verify
\begin{equation} \label{propefe}
 \sum_{n=0}^{N-1} k \|f^{n+1}\|^2 \le \|f\|_{L^2(L^2)}^2.
 \end{equation}
 This holds in particular when $f^{n+1}$ is the average value of $f$ in $(t_n,t_{n+1})$. Moreover, we consider a piecewise affine discretization. We assume that the domain $\Omega$ is polygonal and consider a triangulation ${\cal T}_h$ of $\Omega$. The discretization space for the large scales is defined as
$$
X_h=\{ v_h \in C^0(\overline{\Omega})\,|\, {v_h}_K \in P_1(K),\,\forall K \in {\cal T}_h\, \}.
$$
We further  assume that $h/k=A$ for some constant $A$. Before presenting the error estimates, let us state an auxiliary result.
\begin{lemma}\label{lemmabeta}
 For $M$ large enough, $\beta_j^{K,n}$ defined in expression (\ref{series}) satisfies that
\begin{equation}\label{cotabeta}
\sum_{j=0}^{M} \beta_j^{K,n}\leq c_\beta\frac{ h^2}{k},
\end{equation}
where $c_\beta$ is a positive constant.
\end{lemma}
\begin{proof}
From expression (\ref{series}),
$
\beta_j^{K,n} =(\Lambda_j^{K,n})^{-1},
$
and from Proposition \ref{propADR}, 
\begin{equation*}
 \Lambda_j^{K,n}=1+k\mu\left(\sigma_j^{(K)}+\frac{|\mathbf{c}_n|^2}{4\mu^2}\right)=1+k\mu \left(\frac{j\pi}{h}\right)^2+k\frac{|\mathbf{c}_n|^2}{4\mu},\quad\forall j\in\mathbb{Z}.
 \end{equation*}
Therefore, for $M$ large enough, taking into account that $\sum_{j=0}^{\infty}\frac{1}{j^2}=\frac{\pi^2}{6},$
\begin{equation*}\label{ineqbeta}
\sum_{j=0}^{M} \beta_j^{K,n}\leq c\sum_{j=0}^{M} \frac{h^2}{j^2 k}\leq c_\beta\frac{ h^2}{k},
\end{equation*}
being $c$ and $c_\beta$ positive constants, which concludes the proof.
\end{proof}
In the sequel the norms without subindex will denote the $L^2$ norm, when this will not be source of confusion. Also for brevity we shall denote by $L^p(L^q)$ the space $L^p((0,T),L^q(\Omega))$ and similarly $L^p (H^k)$.

Now, we proceed to perform the error estimates, beginning with the diffusion dominated flow case.
\subsection{Diffusion dominated regime}
\begin{theorem} 
Let 
\begin{equation}\label{defnu1}
\nu_1=2k\mu-2c_{\beta} h^2k\|\mathbf{c}\|_\infty(p_\infty+2q_\infty)
\end{equation}
and 
\begin{equation}\label{defnu2}
\nu_2=2k\mu-3c_{\beta}h^2\pi^2k\|\mathbf{c}\|_\infty(p_\infty-q_\infty),
\end{equation}
where $p_\infty,$ $q_\infty$ and  $\|\mathbf{c}\|_\infty$ are defined in expression (\ref{definf}) and $c_\beta$ is given in Lemma \ref{lemmabeta},.
Assume that $h,k$ are such that $\nu_1, \nu_2 \le \alpha k \mu,$ for some $\alpha>0$ and that the exact solution of problem \eqref{WEAD} satisfies $U\in H^1(L^2)$. Then, there exists a positive constant $c_d$ independent of $h$ and $k,$ such that,
\begin{equation*}\label{desdiff}
\begin{array}{l}
\|\delta_h\|^2_{L^\infty(L^2)}+\mu\|\delta_h\|^2_{L^2(H^1)}\leq c_d (1+C_dT)e^{4TC}(\|\delta_{h}^0\|^2+h^2 (\|f\|_{L^2(L^2)}^2+\|U_{eh}^0\|^2)),
\end{array}
\end{equation*}
where 
\begin{equation}\label{valorCd}
C_d=4c_{\beta}\pi^2q_\infty.
\end{equation}

\end{theorem}
\begin{proof}
Consider $U_{eh}^{n+1},$ solution of the large scales problem (\ref{ecuVMS}) and $U_{M}^{n+1},$ solution of the of the VMS spectral method (\ref{ecuVMS_M}), that is,
\begin{equation}\label{ecuVMSandM}
\begin{array}{l}
B(U_{eh}^{n+1},V_h)+B(\Pi(R(U_{eh}^{n+1})),V_h)=l^{n+1}(V_h), \quad \forall V_h\in X_h,\\ \noalign{\smallskip}
B(U_{M}^{n+1},V_h)+B(\Pi^M(R(U_{M}^{n+1})),V_h)=l^{n+1}(V_h), \quad \forall V_h\in X_h.\\ \noalign{\smallskip}
\end{array}
\end{equation}
Substracting both equations and defining $\delta_h^{n+1}=U_{eh}^{n+1}-U_M^{n+1},$ we obtain,
\begin{equation*}\label{error2}
\begin{array}{l}
B(\delta_h^{n+1},V_h)+\displaystyle\sum_{K\in\mathcal{T}_h}\sum_{j=0}^{M}\beta_j^{K,t} \langle \delta_h^n+\tilde{\mathcal{L}}^*(-\delta_h^{n+1}),p_K \hat{z}_j^{(K)}\rangle(\tilde{\mathcal{L}}^*V_h,\hat{z}_j^{(K)})+\\ \noalign{\smallskip}
\displaystyle\sum_{K\in\mathcal{T}_h}\sum_{j>M}\beta_j^{K,t} \langle R(U_{eh}^{n+1}),p_K \hat{z}_j^{(K)}\rangle(\tilde{\mathcal{L}}^*V_h,\hat{z}_j^{(K)})=(\delta_h^{n},V_h), \quad \forall V_h\in X_h,\\ \noalign{\smallskip}
\end{array}
\end{equation*}
where $$\tilde{\mathcal{L}}^*_KV_h= V_h-k \mathbf{c}_K\cdot  \nabla V_h -k\mu \Delta V_h=V_h-k \mathbf{c}_k\cdot \nabla V_h,$$
on each element $K\in\mathcal{T}_h.$

Taking $V_h=\delta_h^{n+1},$
\begin{equation}\label{error3}
\begin{array}{l}
(\delta_h^{n+1}-\delta_h^n,\delta_h^{n+1})+ k\mu (\nabla\delta_h^{n+1},\nabla \delta_h^{n+1})
=\displaystyle\sum_{K\in\mathcal{T}_h}\left(T_K^A+T_K^B\right),\\ \noalign{\smallskip}
\end{array}
\end{equation}
where we have denoted
\begin{equation*}
\begin{array}{l}
T_K^A=\displaystyle\sum_{j=0}^{M}\beta_j^{K,t} \langle \tilde{\mathcal{L}}^*\delta_h^{n+1}-\delta_h^n,p_K \hat{z}_j^{(K)}\rangle(\tilde{\mathcal{L}}^*\delta_h^{n+1},\hat{z}_j^{(K)}),\\  \noalign{\smallskip}
T_K^B=-\displaystyle\sum_{j>M}\beta_j^{K,t} \langle R(U_{eh}^{n+1}),p_K \hat{z}_j^{(K)}\rangle(\tilde{\mathcal{L}}^*\delta_h^{n+1},\hat{z}_j^{(K)}).
\end{array}
\end{equation*}
Applying successively Young and Cauchy-Schwarz inequalities we get, 
\begin{equation*}
\begin{array}{rcl}
T_K^A&\leq &\displaystyle \sum_{j=0}^{M}\displaystyle\frac{1}{2}\beta_j^{K,t}\| \delta_h^{n+1}-\delta_h^n+k \mathbf{c}_K\cdot \nabla \delta_h^{n+1}\|^2 \|p_K \hat{z}_j^{(K)}\|^2\\  \noalign{\smallskip}
&&\displaystyle\frac{1}{2}\beta_j^{K,t}\| \delta_h^{n+1}-k \mathbf{c}_K\cdot \nabla \delta_h^{n+1}\|^2 \| \hat{z}_j^{(K)}\|^2
\\  \noalign{\smallskip}
&\leq&\displaystyle\frac{h^2c_{\beta}}{k}\left(\| p_K\|_{\infty}(\| \delta_h^{n+1}-\delta_h^n\|^2+k^2\|\mathbf{c}_K\|_\infty^2\|\nabla \delta_h^{n+1} \|^2)+\right.
\\  \noalign{\smallskip}
&&\left.\displaystyle \| p_K^{-1}\|_{\infty}(\| \delta_h^{n+1}\|^2+k^2\|\mathbf{c}_K\|_\infty^2\|\nabla \delta_h^{n+1} \|^2)\right),
\end{array}
\end{equation*}
where in the second step, we have applied the facts that functions $\hat{z}_j^{(K)}$ form an orthonormal system in $L^2_{p_K}(K)$ and
Lemma \ref{lemmabeta}.

Working in a similar manner with $T_B^K,$ we find that
\begin{equation*}
\begin{array}{rcl}
T_K^B &\leq&\displaystyle\frac{h^2c_{\beta}}{k}\left(\frac{3}{2}\| p_K\|_{\infty}(k^2\|f_K^{n+1}\|^2+k^2\left\|\frac{ U_{eh}^{n+1}-U_{eh}^n}{k}\right\|^2+k^2 \|\mathbf{c}_K\|_\infty^2\|\nabla U_{eh}^{n+1} \|^2)+\right.
\\  \noalign{\smallskip}
&&\left.\displaystyle \| p_K^{-1}\|_{\infty}(\| \delta_h^{n+1}\|^2+k^2\|\mathbf{c}_K\|_\infty^2\|\nabla \delta_h^{n+1} \|^2)\right).
\end{array}
\end{equation*}

Bearing in mind that $(\delta_h^{n+1}-\delta_h^n,\delta_h^{n+1})=(\|\delta_h^{n+1}\|^2-\|\delta_h^{n}\|^2+\|\delta_h^{n+1}-\delta_h^{n}\|^2)/2,$ 
from expression (\ref{error3}), we obtain
\begin{equation}\label{ineqdiff1}
\left(1-C_dk\right)\|\delta_h^{n+1}\|^2+\nu_1\|\nabla\delta_h^{n+1}\|^2\leq \|\delta_h^n\|^2+\theta_n,
\end{equation}
where $\nu_1$ is defined in equation (\ref{defnu1}), $C_d$ is given in expression (\ref{valorCd})
and 
\begin{equation*}\label{theta}
\theta_n=\displaystyle 2c_{\beta}h^2p_{\infty}k\|f^{n+1}\|^2+6c_{\beta} h^2p_\infty k \left\|\frac{ U_{eh}^{n+1}-U_{eh}^n}{k}\right\|^2+3c_{\beta}h^2k \|\mathbf{c}\|_\infty^2p_\infty \|\nabla U_{eh}^{n+1}\|^2.
\end{equation*}
Now, assuming that $h,k$ are such that $\nu_1\simeq k\mu,$ it is possible to apply the discrete Gronwall lemma and obtain,

\begin{equation}\label{desdiff}
\begin{array}{rcl}
\displaystyle\max_{1\leq n\leq N}\|\delta_h^{n}\|^2&\leq &       \displaystyle e^{4TC_d}\left(\|\delta_{h}^0\|^2+4c_{\beta}h^2p_\infty\|f\|_{L^2(L^2)}^2 +6c_{\beta}h^2p_\infty\sum_{n=0}^{N-1}k \left\|\frac{ U_{eh}^{n+1}-U_{eh}^n}{k}\right\|^2\right.\\ \noalign{\smallskip}
&&\displaystyle\left.6c_{\beta}h^2\|\mathbf{c}\|_\infty p_\infty\|U_{eh}^{n+1}\|_{L^2(H^1)}^2\right).
\end{array}
\end{equation}
As from assumption $U\in H^1(L^2),$ it is standard that there exists some constant $c_e$ such that 
\begin{equation}\label{cotaUdif}
\sum_{n=0}^{N-1}k \left\|\frac{ U_{eh}^{n+1}-U_{eh}^n}{k}\right\|^2<c_e.
\end{equation}
Now, to obtain an appropriate bound for $\| U_{eh}^{n+1}\|^2_{L^2(H^1)},$ it is necessary to apply similar reasoning to that used in this proof so far, to the first equation of expression (\ref{ecuVMSandM}), from where we obtain

\begin{equation}\label{ineqdiffU}
\left(1-\tilde{C}_dk\right)\|U_{eh}^{n+1}\|^2+\nu_2\|\nabla U_{eh}^{n+1}\|^2\leq \|U_{eh}^n\|^2+(1+3c_{\beta}h^2p_\infty)k\|f^{n+1}\|^2,
\end{equation}
where 
\begin{equation*}\label{tildeCd}
\begin{array}{rcl}
\displaystyle\tilde{C}_d=1+\frac{2c_{\beta}h^2q_\infty}{k^2}
\end{array}
\end{equation*}
and $\nu_2$ is defined in equation (\ref{defnu2}).
Now, assuming that $h,k$ are such that $\nu_2\simeq k\mu,$ it is possible to apply the discrete Gronwall lemma and obtain,
\begin{equation}\label{desdiffU}
\begin{array}{rcl}
\displaystyle\max_{1\leq n\leq N}\|U_{eh}^{n}\|^2&\leq &       \displaystyle e^{4T\tilde{C_d}}\left(\|U_{eh}^0\|^2+2\left(1+3c_{\beta}h^2p_\infty\right)\|f\|_{L^2(L^2)}^2\right).\\ \noalign{\smallskip}
\end{array}
\end{equation}
Summing up with respect to $n$ in inequality (\ref{ineqdiffU}), it follows that,
\begin{equation}\label{cotaL2H1U}
\begin{array}{l}
\|U_{eh}^N\|^2+\mu \| U_{eh}\|^2_{L^2(H^1)}\leq \displaystyle(1+\tilde{C}_dT) \max_{1\leq n\leq N}\|U_{eh}^{n}\|^2+\left(1+3c_{\beta}h^2p_\infty\right)\|f\|^2_{L^2(L^2)} \\ \noalign{\smallskip}
\leq \displaystyle(1+\tilde{C}_dT) e^{4T\tilde{C}_d}\|U_{eh}^{0}\|^2+\left(1+3c_{\beta}h^2 p_\infty\right)\left(2(1+\tilde{C}_dT) e^{4T\tilde{C}_d}+1\right)\|f\|^2_{L^2(L^2)}:=A,
\end{array}
\end{equation}
where in the second inequality we have taken into account inequality (\ref{desdiffU}).Therefore, 
\begin{equation}\label{cotaL2H1U2}
\begin{array}{l}
 \| U_{eh}\|^2_{L^2(H^1)}\leq A/\mu,
\end{array}
\end{equation}
where $A$ is defined in (\ref{cotaL2H1U}). 

Finally, summing up with respect to $n$ in inequality (\ref{ineqdiff1}), bearing in mind that $\nu_1\simeq k\mu$ and inequalities (\ref{desdiff}), (\ref{cotaUdif}) and  (\ref{cotaL2H1U2}), the conclusion follows. 

\end{proof}
Now, we proceed to analyze the advection dominated flow case.
\subsection{Advection dominated regime}
To better understand the effect of adding the stabilizing term into equation (\ref{ecuVMS_M}) in this case,  we proceed to its computation through the Green's function technique.

From 
\cite{brezzi}, the stabilizing term in equation (\ref{ecuVMS}) can also be written as
\begin{equation}\label{greenbrezzi}
B_n(\Pi_n(R_n(U_h^{n+1})),V_h)=\displaystyle\sum_K\int_{K\times K}g_y^{(n,K)}(x)(R_n(U_h^{n+1}))(y)(\tilde{\mathcal{L}}_{n,K}^* V_h)(x)dxdy,
\end{equation}
where $R_n(U_h^{n+1})$ is the residual defined in expression (\ref{residuo}),     $\tilde{\mathcal{L}}_{n,K}^*$ is the formal adjoint of $\tilde{\mathcal{L}_n}_K$ (with zero boundary conditions on $\partial K$), and for $y\in K,$ $g_y^{(n,K)}$ is the element Green's function of our problem,
\begin{equation*}\label{greenpbm}
\left\{\begin{array}{ll}
\tilde{\mathcal{L}_n}_K g_y^{(n,K)}=\delta_y & \mbox{in }K,\\ \noalign{\smallskip}
g_y^{(n,K)}=0& \mbox{on }\partial K.
\end{array}\right.
\end{equation*}

On the other hand, $\tilde{U}^{n+1}_K$ is the solution of
\begin{equation*}\label{Utildepbm}
\left\{\begin{array}{ll}
\tilde{\mathcal{L}_n}_K \tilde{U}^{n+1}_K=R_n(U_h^{n+1}) & \mbox{in }K,\\ \noalign{\smallskip}
 \tilde{U}^{n+1}_K=0& \mbox{on }\partial K.
\end{array}\right.
\end{equation*}
From Theorem \ref{maintheorem},  we can develop
\begin{equation*}\label{Udevelop}
\tilde{U}_K^{n+1}(x)=\displaystyle\sum_{j=0}^{\infty}\beta_j^{K,n} \langle R_n(U_h^{n+1}),p_{n,K} \hat{z}_j^{(n,K)}\rangle \hat{z}_j^{(n,K)}(x).
\end{equation*}
Hence, an adaptation of Theorem  \ref{maintheorem} just changing $R_n(U_h^{n+1})$ by $\delta_y,$ allows us to write the Green function in terms of the eigenfunctions of operator $\tilde{\mathcal{L}_n}_K$ as
\begin{equation}\label{greenseries}
g_y^{(n,K)}(x) =\displaystyle\sum_{j=0}^{\infty}\beta_j^{K,n} \langle \delta_y,p_{n,K} \hat{z}_j^{(n,K)}\rangle \hat{z}_j^{(n,K)}=\displaystyle\sum_{j=0}^{\infty}\beta_j^{K,n} \left(p_{n,K} \hat{z}_j^{(n,K)}\right)(y) \hat{z}_j^{(n,K)}(x).
\end{equation}
We can think of the spectral method (\ref{ecuVMS}) as an alternative form of computing the Green's function, and of the feasible discretization (\ref{ecuVMS_M}), as a truncation of the series (\ref{greenseries}), namely
\begin{equation}\label{greenseriestrun}
g_y^{(n,K,M)}(x) =\displaystyle\sum_{j=0}^{M}\beta_j^{K,n} \left(p_{n,K} \hat{z}_j^{(n,K)}\right)(y) \hat{z}_j^{(n,K)}(x).
\end{equation}
Note that function $g_y^{(n,K,M)}$ corresponds to the solution of
\begin{equation*}\label{greenpbmM}
\left\{\begin{array}{ll}
\tilde{\mathcal{L}_n}_K g_y^{(n,K,M)}=\delta_y^M & \mbox{in }K,\\ \noalign{\smallskip}
g_y^{(n,K,M)}=0& \mbox{on }\partial K,
\end{array}\right.
\end{equation*}
where we define
\begin{equation*}
(\delta^M_y,V)=\langle \delta_y,V \rangle, \quad\forall V\in \mathcal{V}_M,
\end{equation*}
with $ \mathcal{V}_M=\mathcal{L}_n\{ \hat{z}_1^{(n,K)},\ldots,  \hat{z}_M^{(n,K)}\}.$

In the following lemma, we prove that $g_y^{(n,K,M)}\in L^2(K\times K).$
\begin{lemma}\label{lemagreen}
Function $g_y^{(n,K,M)}$ defined in expression (\ref{greenseriestrun}) satisfies that $g_y^{(n,K,M)}\in L^2(K\times K)$ and
\begin{equation}\label{cotag}
\displaystyle\int_{K\times K}|g_y^{(n,K,M)}(x)|^2dxdy\leq \|p_{n,K}^{-1}\|_{\infty}  \|p_{n,K}\|_{\infty} \frac{\pi^4 h^4}{36k^2}.
\end{equation}
\end{lemma}
\begin{proof}
Taking into account Cauchy-Schwarz inequality, and the fact that each summand of function $g_y^{(n,K,M)}(x)$ is a product of two functions depending on different variables, it follows that,
\begin{equation}\label{ineq1}
\begin{array}{l}
\displaystyle\int_{K\times K}|g_y^{(n,K,M)}(x)|^2dxdy=\displaystyle\int_{K\times K}\displaystyle \left|\sum_{j=0}^{M}\beta_j^{K,n} \left(p_{n,K} \hat{z}_j^{(n,K)}\right)(y) \hat{z}_j^{(n,K)}(x)\right|^2dxdy\leq\\ \noalign{\smallskip}
\displaystyle\int_K\sum_{j=0}^{M} \beta_j (p_{n,K}  \hat{z}_j^{(n,K)})^2(y)\left(\int_K\sum_{j=0}^{M} \beta_j  (\hat{z}_j^{(n,K)})^2(x)dx\right)dy.\\ \noalign{\smallskip}
\end{array}
\end{equation}
Now, bearing in mind that $\hat{z}_j^{(n,K)}$ is an orthonormal system in $L^2_{p_{n,K}}(K),$ it holds,
\begin{equation}\label{ineq2}
\displaystyle \int_K  (p_{n,K}  \hat{z}_j^{(n,K)})^2(y)dy\leq \|p_{n,K}\|_\infty,\quad\mbox{and}\quad\displaystyle \int_K (\hat{z}_j^{(n,K)})^2(x)dx\leq \|p_{n,K}^{-1}\|_\infty.
\end{equation}  
From inequalities (\ref{ineq1}), (\ref{ineq2}) and Lemma \ref{lemmabeta}, the bound (\ref{cotag}) follows. 
\end{proof}

Next, we deal with the pointwise convergence of series defined in (\ref{greenseriestrun}) to series (\ref{greenseries}).
\begin{lemma}
It is satisfied that
\begin{equation*}
g_y^{(n,K,M)}(x)\xrightarrow{M\rightarrow\infty}g_y^{(n,K)}(x),\quad \forall (x,y)\in K\times K-\{x=y\}.
\end{equation*}
\end{lemma}
\begin{proof}
Let us define 
\begin{equation*}
s(x,y)=\sum_{j=0}^\infty s_j(x,y),
\end{equation*}
where $s_j(x,y)=\beta_j^{K,n}(p_{n,K}\hat{z}_j^{(n,K)})(y)\hat{z}_j^{(n,K)}(x).$ From Lemma \ref{lemagreen},
\begin{equation*}
\sum_{j=0}^M s_j(x,y)\xrightarrow{M\rightarrow\infty}s(x,y),\quad \forall(x,y)\in K\times K-\mathcal{A},
\end{equation*}
where $\mathcal{A}$ in a null measure set. On the other hand, from the pointwise convergence of the Fourier series \cite{libroharmonic} and the continuity of $s,$ for fixed $y,$
\begin{equation*}
\sum_{j=0}^M s_j(x,y)\xrightarrow{M\rightarrow\infty}g_y(x),\quad \forall(x,y)\in K\times K-\mathcal{B},
\end{equation*}
where $\mathcal{B}=\{x=y\}.$ Therefore, $s(x,y)=g_y(x)$ for $(x,y)\in K\times K-\{\mathcal{A}\cup\mathcal{B}\}.$

Finally, for $(x,y)\in\mathcal{A}-\mathcal{B},$ there exists a sequence $(x_n,y_n)\in K\times K-\{\mathcal{A}\cup\mathcal{B}\}$ such that $(x_n,y_n)\rightarrow(x,y)$ as $n\rightarrow\infty,$ and $s(x_n,y_n)=g_{y_n}(x_n)\rightarrow g_y(x)$ as $n\rightarrow\infty.$ Hence, there exists the limit when $n\rightarrow\infty$ of $s(x_n,y_n)$ and it equals $g_y(x).$

\end{proof}

At this point, we would like to understand the way the stabilizing terms modify the Galerkin formulation. For the sake of simplicity, we will assume that 
\begin{equation*}\begin{array}{l}
\tilde{\mathcal{L}_n}_KU_h^{n+1}\simeq \bar{U}_h^{n+1}+k \mathbf{c}_{n,K} \cdot \nabla U_h^{n+1}-k\mu \Delta U_h^{n+1}:=\bar{\mathcal{L}_n}_KU_h^{n+1},\quad \\  \noalign{\smallskip}
l^{n+1}(V_h)\simeq k \langle f_K,V_h\rangle+(\bar{U}^{n}_h,V_h):=\bar{l}^{n+1}(V_h), \\  \noalign{\smallskip}
\end{array}
\end{equation*}
where  $\bar{U}_h^{n},$ $\bar{U}_h^{n+1}$ and $\bar{V}_h$ are piecewise constant  on each element $K,$ and we will denote
\begin{equation*}\label{barR}
\bar{R}(U_h^{n+1})=\bar{l}^{n+1}- \bar{\mathcal{L}_n}_KU_h^{n+1}.
\end{equation*}
If there is no place to confusion, we will omit the bar signs, for the sake of clarity. Thus, from equation (\ref{greenbrezzi}), we can writte
\begin{equation}\label{greenbrezzi2}
B_n(\Pi_n(R_n(U_h^{n+1})),V_h)=\displaystyle\sum_K\hat{\tau}_K \left[\int_{K}(R_n(U_h^{n+1}))(\tilde{\mathcal{L}_n}^*_K V_h)\right],
\end{equation}
where, 
\begin{equation}\label{tauk}
 \hat{\tau}_{n,K}=  \frac{1}{|K|}\int_{K\times K}g_y^{(n,K)}(x)dxdy= \frac{1}{|K|}\int_K\left[\int_K g_y^{(n,K)}(x) 1dx\right] dy= \frac{1}{|K|}\int_Kb_{n,K},
 \end{equation}
 
 being $b_{n,K}$ the \lq\lq bubble function"solution of the initial value problem in $K,$
 \begin{equation}\label{bk}
\left\{\begin{array}{ll}
b_{n,K}+k \mathbf{c}_{n,K}\cdot \nabla b_{n,K}-k \mu \Delta b_{n,K}=1&\mbox{in }K,\\ \noalign{\smallskip}
b_{n,K}=0 &\mbox{on }\partial K.
\end{array}\right.
\end{equation}
We compute the solution of problem (\ref{bk}) and the explicit expression of $\hat{\tau}_K$ in the 1D case in Appendix \ref{secapp}.

Let us define 
\begin{equation}
 \hat{\tau}_{n,K}^M\label{taukM}=\frac{1}{|K|}\int_{K\times K}g_y^{(n,K,M)}(x)dxdy.
 \end{equation}
 
If we consider a piecewise affine discretization of problem (\ref{WEAD}),  then,
\begin{equation*}
\tilde{\mathcal{L}_n}_KU_h^{n+1}=U_h^{n+1}+k \mathbf{c}_{n,K} \cdot \nabla U_h^{n+1},\quad
\tilde{\mathcal{L}_n}^*_KV_h=V_h-k \mathbf{c}_{n,K}\cdot  \nabla V_h
\end{equation*}
and from expression (\ref{greenbrezzi}), we can write
\begin{equation}\label{S}
B_n(\Pi_n(R_n(U_h^{n+1})),V_h)=S_{\mathcal{L}_n}(U_h^{n+1},V_h)-S_{l}(V_h),
\end{equation}
where $S_{\mathcal{L}_n}$ includes the stabilizing terms at the left-hand side and $S_{l}$ includes the stabilizing terms at the right-hand side, namely,
\begin{equation*}\label{SL}
\begin{array}{r}
S_{\mathcal{L}_n}(U_h^{n+1},V_h)= \displaystyle\sum_{K\in\mathcal{T}_h} \hat{\tau}_{n,K} \left[\int_Kk^2(\mathbf{c}_{n,K}\cdot\nabla U_h^{n+1})(\mathbf{c}_{n,K}\cdot\nabla V_h)+k\mathbf{c}_{n,K} \cdot\nabla V_h {U}_h^{n+1} \right .\\
\left . -k\mathbf{c}_{n,K} \cdot\nabla U_h^{n+1}{V}_h-\displaystyle \int_K  {U}_h^{n+1}{V}_h \right ].\\ \noalign{\smallskip}
\end{array}
\end{equation*}
Note that, the term multiplied by $k^2$ in last expression corresponds to the stabilizing term in the stationary advection-diffusion equation \cite{ChaconDia}.
And the stabilizing terms in the right-hand side read,
\begin{equation*}\label{Sl}
\begin{array}{rcl}
S_{l}^{n+1}(V_h)&=& \displaystyle\sum_{K\in\mathcal{T}_h} \hat{\tau}_{n,K}  \left[\int_Kk^2f_K^{n+1}\mathbf{c}_{n,K}\cdot\nabla V_h-kf_K^{n+1}{V}_h+k\mathbf{c}_{n,K} \cdot\nabla V_h {U}_h^{n}-{U}_h^{n}{V}_h\right]. 
\end{array}
\end{equation*}
Thus, the VMS spectral method (\ref{ecuVMS}) can be seen as the Galerkin solution of the modified problem
\begin{equation}\label{stab}
\begin{array}{l}
(B+S_{\mathcal{L}_n})(U_h^{n+1},V_h)=(l^{n+1}+S_{l}^{n+1})(V_h), \quad \forall V_h\in X_h,\\ \noalign{\smallskip}
n=0,1,...,N-1.
\end{array}
\end{equation}
From a practical point of view, this allows to write the feasible discretization (\ref{ecuVMS_M}) as
\begin{equation}\label{APVMS}
\begin{array}{l}
(1- \hat{\tau}^M )[(U_{M}^{n+1},V_h)+k(\mathbf{c}_n\cdot\nabla U_{M}^{n+1},V_h)]+k\mu(\nabla U_{M}^{n+1},\nabla V_h)+
\hat{\tau}^M k (U_{M}^{n+1},\mathbf{c}_n\cdot\nabla V_h) \\  \noalign{\smallskip}
+\hat{\tau}^M k^2(\mathbf{c}_n\cdot\nabla U_{M}^{n+1},\mathbf{c}_n\cdot\nabla V_h)\\  \noalign{\smallskip}=
(1- \hat{\tau}^M )[k\langle  f^{n+1},V_h\rangle+(U_{M}^{n},V_h)]+\hat\tau^M k^2\langle  f^{n+1},\mathbf{c}_n\cdot\nabla V_h\rangle+\hat{\tau} k (U_{M}^n,\mathbf{c}_n\cdot \nabla V_h),
\end{array}
\end{equation}
where $\hat{\tau}^M\rightarrow\hat{\tau},$  when the number of eigenpairs $M\rightarrow\infty.$

Taking advantage of this new formulation, we present the result about the error analysis in the convection dominated flow case.

For the sake of simplicity, we consider a uniform mesh and denote $\hat\tau=\hat\tau_K$ and $\hat\tau^{M}=\hat\tau_K^M,$ where $\hat\tau_K,$ $\hat\tau_K^{M}$ are given in expressions (\ref{tauk}) and (\ref{taukM}), respectively. We shall assume that $\tau_K $ is of order $k$, this is proved in the 1D case in the Appendix.
\begin{theorem} \label{thmAD}
There exist a positive constant $c_a$ independent of $h,$ $k,$ such that,	
\begin{equation*}
\begin{array}{l}\label{estcdom}
\|\delta_h\|^2_{L^\infty(L^2)}+\mu\|\delta_h\|^2_{L^2(H^1)}\leq (1+C_aT)e^{4TC_a}\|\delta_{h}^0\|^2+c_a |\hat\tau-\hat\tau^M| (\|f\|_{L^2(L^2)}^2+\|U_{eh}^0\|^2),
\end{array}
\end{equation*}
where 
\begin{equation}\label{valorCa}
C_a=\frac{\hat\tau^M+|\hat\tau-\hat\tau^M|(2k+1)}{k}.
\end{equation}
\end{theorem}
\begin{proof}

Bearing in mind the analysis done so far, in particular expressions (\ref{S})-(\ref{stab}), the VMS method (\ref{ecuVMS}) can be explicitly written as
\begin{equation}\label{EXVMS}
\begin{array}{l}
(1- \hat{\tau} )[(U_{eh}^{n+1},V_h)+k(\mathbf{c}_n\cdot\nabla U_{eh}^{n+1},V_h)]+k\mu(\nabla U_{eh}^{n+1},\nabla V_h)+
\hat{\tau} k (U_{eh}^{n+1},\mathbf{c}_n\cdot\nabla V_h) \\  \noalign{\smallskip}
+\hat{\tau}k^2(\mathbf{c}_n\cdot\nabla U_{eh}^{n+1},\mathbf{c}_n\cdot\nabla V_h)\\  \noalign{\smallskip}=
(1- \hat{\tau} )[k\langle  f^{n+1},V_h\rangle+(U_{eh}^{n},V_h)]+\hat\tau k^2\langle  f^{n+1},\mathbf{c}_n\cdot\nabla V_h\rangle 
+\hat{\tau} k (U_{eh}^n,\mathbf{c}_n\cdot \nabla V_h)
\end{array}
\end{equation}

Substracting expression (\ref{APVMS}) from expression (\ref{EXVMS}),  we obtain
\begin{equation*}\label{subs}
\begin{array}{l}
(1-\hat{\tau}^M)(\delta_h^{n+1},V_h)+(1-\hat{\tau}^M)k(\mathbf{c}_n\cdot\nabla\delta_h^{n+1},V_h)+ k\mu (\nabla\delta_h^{n+1},\nabla V_h)+\\\noalign{\smallskip}
\hat{\tau}^M k (\delta_h^{n+1},\mathbf{c}_n\cdot\nabla V_h)
+\hat{\tau}^M k^2(\mathbf{c}_n\cdot\nabla\delta_h^{n+1},\mathbf{c}_n\cdot \nabla V_h) =(1-\hat{\tau}^M)(\delta_h^n,V_h)+\hat\tau^M k (\delta_h^n,\mathbf{c}_n\cdot\nabla V_h)+\\\noalign{\smallskip}
(\hat{\tau}-\hat{\tau}^M)[(U_{eh}^{n+1}-U_{eh}^{n},V_h)+k(\mathbf{c}_n\cdot\nabla U_{eh}^{n+1},V_h)-k(U_{eh}^{n+1}-U_{eh}^n,\mathbf{c}_n\cdot \nabla V_h)\\\noalign{\smallskip}
-k^2 (\mathbf{c}_n\cdot \nabla U_{eh}^{n+1},\mathbf{c}_n\cdot\nabla V_h)-k \langle f^{n+1},V_h\rangle+k^2\langle f^{n+1},\mathbf{c}_n\cdot \nabla V_h\rangle].
\end{array}
\end{equation*}
Taking $V_h=\delta_h^{n+1},$
\begin{equation}\label{subs2}
(1-\hat{\tau}^M)(\delta_h^{n+1}-\delta_h^n,\delta_h^{n+1})+ k\mu \|\nabla\delta_h^{n+1}\|^2
+\hat{\tau}^M k^2\|\mathbf{c}_n\cdot\nabla\delta_h^{n+1}\|^2 =\hat\tau^M k (\delta_h^n,\mathbf{c}_n\cdot\nabla \delta_h^{n+1})+\rho_n, 
\end{equation}
where 
\begin{equation*}\label{defrho}
\begin{array}{rcl}
\rho_n&=& (\hat{\tau}-\hat{\tau}^M)[(U_{eh}^{n+1}-U_{eh}^{n},\delta_h^{n+1})+k(\mathbf{c}_n\cdot\nabla U_{eh}^{n+1},\delta_h^{n+1})-k(U_{eh}^{n+1}-U_{eh}^n,\mathbf{c}_n\cdot \nabla \delta_h^{n+1})-\\\noalign{\smallskip}
&&k^2 (\mathbf{c}_n\cdot \nabla U_{eh}^{n+1},\mathbf{c}_n\cdot\nabla \delta_h^{n+1})
-k \langle f^{n+1},\delta_h^{n+1}\rangle+k^2\langle f ^{n+1},\mathbf{c}_n\cdot \nabla \delta_h^{n+1}\rangle].
\end{array}
\end{equation*}

By using Cauchy-Schwarz and Young inequalities successively, we get that
\begin{equation}\label{cotarho}
\begin{array}{l}
\displaystyle|\rho_n|\leq |\hat{\tau}-\hat{\tau}^M|\left[(k+1)\frac{\|U_{eh}^{n+1}-U_{eh}^n\|^2}{2}+(2k+1)\frac{\|\delta_{h}^{n+1}\|^2}{2}+k(k+1)\frac{\|\mathbf{c}_n\cdot\nabla U_{eh}^{n+1}\|^2}{2}\right.
\\\noalign{\smallskip}
\displaystyle\left. +k(2k+1)\frac{\|\mathbf{c}_n\cdot\nabla \delta_{h}^{n+1}\|^2}{2}+k(k+1)\frac{\|f^{n+1}\|^2}{2}\right].
\end{array}
\end{equation}
On the other hand, using Young inequality with constant $\varepsilon$ to be chosen, 
\begin{equation}\label{youngine}
\hat\tau^M k (\delta_h^n,\mathbf{c}_n\cdot\nabla \delta_h^{n+1})\leq \hat\tau^M k\left(\frac{\|\delta_h^n\|^2}{2\varepsilon}+\frac{\varepsilon\|\mathbf{c}_n\cdot\nabla\delta_h^{n+1}\|^2}{2}\right).
\end{equation}
 Taking into account that 
$$(\delta_h^{n+1}-\delta_h^n,\delta_h^{n+1})=(\|\delta_h^{n+1}\|^2-\|\delta_h^{n}\|^2+\|\delta_h^{n+1}-\delta_h^{n}\|^2)/2,
$$  
combining expression (\ref{subs2}) with (\ref{cotarho}) and (\ref{youngine}) by choosing $\varepsilon=k$, we obtain,

\begin{equation}\label{ineqdiff2}
(1-k C_a)\|\delta_h^{n+1}\|^2+\|\delta_h^{n+1}-\delta_h^{n}\|^2+2k\mu\|\nabla\delta_h^{n+1}\|^2\leq \|\delta_h^n\|^2+\beta_n,
\end{equation}
where $C_a$ is given in expression (\ref{valorCa}) and
\begin{equation*}\label{beta}
\beta_n=|\hat{\tau}-\hat{\tau}^M|(k+1)\left[\|U_{eh}^{n+1}-U_{eh}^n\|^2+k\|\mathbf{c}_n\cdot\nabla U_{eh}^{n+1}\|^2+k\|f^{n+1}\|^2].\right.
\end{equation*}
As $\tau^M$ is of order $k$, it is possible to apply the discrete Gronwall lemma for $k$ small enough to obtain
\begin{equation}\label{desdiff2}
\begin{array}{rcl}
\displaystyle\max_{1\leq n\leq N}\|\delta_h^{n}\|^2\leq        & e^{4TC_a}&\displaystyle\left[\|\delta_{h}^0\|^2+2|\hat\tau-\hat\tau^M|(k+1)\left(\sum_{n=0}^{N-1}\|U_{eh}^{n+1}-U_{eh}^n\|^2\right.+\right.
\\\noalign{\smallskip}
 &  &  \displaystyle\left.\left.\|\mathbf{c}_n\|_\infty \,\sum_{n=0}^{N-1} k \|U_{eh}^{n+1}\|^2+\|f\|^2_{L^2(L^2)}\right)\right].
\end{array}
\end{equation}

Now, to obtain appropriate bounds for $\displaystyle \sum_{n=0}^{N-1}\|U_{eh}^{n+1}-U_{eh}^n\|^2$ and $\displaystyle \sum_{n=0}^{N-1}\|U_{eh}^{n+1}\|^2,$ it is necessary to apply similar reasoning to that used in this proof so far, to equation (\ref{EXVMS}), from where we obtain

\begin{equation}\label{ineqdiffU2}
\left(1-\tilde{C}_ak\right)\|U_{eh}^{n+1}\|^2+(1-\hat\tau)\|U_{eh}^{n+1}-U_{eh}^n\|^2+2k\mu\|\nabla U_{eh}^{n+1}\|^2 \leq \|U_{eh}^n\|^2+(1-\hat\tau+2k)k\|f\|^2,
\end{equation}
where 
\begin{equation*}\label{tildeCa}
\begin{array}{rcl}
\displaystyle\tilde{C}_a=\frac{\hat\tau+k(1-\hat\tau)}{ k}.
\end{array}
\end{equation*}
Now,  it is possible to apply the discrete Gronwall lemma and obtain,
\begin{equation}\label{desdiffU2}
\begin{array}{rcl}
\displaystyle\max_{1\leq n\leq N}\|U_{eh}^{n}\|^2&\leq &       \displaystyle e^{4T\tilde{C_a}}\left(\|U_{eh}^0\|^2+2(1-\hat\tau+2k)\|f\|_{L^2(L^2)}^2\right).\\ \noalign{\smallskip}
\end{array}
\end{equation}
Summing up with respect to $n$ in inequality (\ref{ineqdiffU2}), it follows that,
\begin{equation}\label{cotaL2H1U2n}
\begin{array}{l}
\displaystyle\|U_{eh}^N\|^2+(1-\hat\tau)\sum_{n=0}^{N-1}\|U_{eh}^{n+1}-U_{eh}^n\|^2+2\mu\,  \sum_{n=0}^{N-1}\|U_{eh}^{n+1}\|^2\\ \noalign{\smallskip}
\leq \displaystyle(1+\tilde{C}_aT) e^{4T\tilde{C_a}}\|U_{eh}^{0}\|^2+(1-\hat\tau+2k)\left(2(1+\tilde{C_a}T) e^{4T\tilde{C_a}}+1\right)\|f\|^2_{L^2(L^2)}:=B,
\end{array}
\end{equation}
where in the second inequality we have taken into account inequality (\ref{desdiffU2}).Therefore, 
\begin{equation}\label{cotaL2H1U22n}
\begin{array}{l}
\displaystyle\sum_{n=0}^{N-1}\|U_{eh}^{n+1}-U_{eh}^n\|^2\leq\frac{B}{1-\hat\tau} \quad\mbox{and}\quad  \sum_{n=0}^{N-1}\|U_{eh}^{n+1}\|^2\leq \frac{B}{2\mu},
\end{array}
\end{equation}
where $B$ is defined in (\ref{cotaL2H1U2n}).  Summing up with respect to $n$ in inequality (\ref{ineqdiff2}), bearing in mind inequalities (\ref{desdiff2}) and  (\ref{cotaL2H1U22n}), the conclusion follows. 
\end{proof}

In Appendix \ref{secapp}, some estimates for $|\hat\tau-\hat\tau^M|$ are derived.

\section{Numerical Results}\label{numerics}
In this section, we present some numerical tests to illustrate the way the spectral method works. We first focus on the 1D stationary advection-diffusion-reaction equation and after that we
proceed to the 1D evolutive advection-diffusion equation. 

\subsection{Stationary advection-diffusion-reaction equation}
Let us consider the following boundary problem for the advection-diffusion-reaction equation,
\begin{equation*}\label{1DEADR}
\left\{\begin{array}{l}
\gamma \bar{U}+c\partial_x \bar{U}-\mu \partial_{xx} \bar{U}=0,\quad \mbox{for }x\in(0,1),\\  \noalign{\smallskip}
\bar{U}(0)=0 \quad \mbox{and}\quad \bar{U}(1)=1.\\ 
\end{array}\right.
\end{equation*}
As the equation is linear, the exact solution is known, namely,
$$\bar{U}(x)=\frac{e^{\frac{1}{2} (c/\mu - \rho) (x-1)} (e^{\rho x}-1)}{e^{\rho}-1},\quad\mbox{with}\quad\rho=\frac{\sqrt{c^2 + 4 \gamma \mu}}{\mu}.$$
Consider the change of variable,
$$U(x)=\bar{U}(x)-x,\quad\mbox{for }x\in[0,1].$$
Thus, problem (\ref{1DEADR}) is rewritten as a new problem with Dirichlet homogeneous boundary conditions,
\begin{equation}\label{1DEADR}
\left\{\begin{array}{l}
\gamma {U}+c\partial_x {U}-\mu \partial_{xx} {U}=f, \quad \mbox{for }x\in(0,1),\\  \noalign{\smallskip}
{U}(0)=0 \quad \mbox{and}\quad {U}(1)=0,\\ 
\end{array}\right.
\end{equation}
where $f(x)=-\gamma x-c.$ Consider uniformly spaced nodes $0=x_1<...<x_{r+1}=1$ and $\{\varphi_{K=1}^{r+1}\}$ the piecewise affine basis functions associated to these nodes.
We look for $\displaystyle {U}_h=\sum_{m=2}^r U^m\varphi_m,$ with $U^m\in\mathbb{R}$ solution of the VMS method with spectral approximation of the sub-grid scales \cite{ChaconDia}. Thus, $\mathbf{U}=(U^1,...U^r)^T$  is the solution of the linear system
$$A\mathbf{U}=\mathbf{b},$$
where $A\in\mathbb{R}^{(r-1)\times(r-1)}$ and $\mathbf{b}\in\mathbb{R}^{r-1}$ are defined as
$$A=\gamma A^R+c A^C+\mu A^D+A^S,\quad \mathbf{b}=\mathbf{b}^1+\mathbf{b}^S,$$
where $A^R,A^C,A^D$ and $A^S$ are respectively, the reaction, convection, diffusion and sub-grid matrices, $\mathbf{b}^1$ is the independent term and $\mathbf{b}^S$ is the stabilized independent term, defined by
\begin{equation*}
\begin{array}{ll}
A^R_{lm}=(\varphi_l,\varphi_m),&A^C_{lm}=(\varphi_l',\varphi_m),\quad\mbox{for}\quad l,m=1,...,r,\\ \noalign{\smallskip}
A^D_{lm}=(\varphi_l',\varphi_m'),&A^S=-\gamma^2\,B^{S1}-c \gamma\,B^{S2}+c \gamma\,B^{S3}+c^2\,B^{S4},\quad\mbox{for} \quad l,m=1,...,r,\\ \noalign{\smallskip}
b^1_{l}=-h(\gamma l h+a),&b_l^S=b_l^{S1}+b_l^{S2},\quad\mbox{for} \quad l,m=1,...,r,\\ \noalign{\smallskip}
\end{array}
\end{equation*}
where the stabilization matrix $A^S$ is given by
\begin{equation}\label{expB}
\begin{array}{l}
B^{S1}_{lm}=\displaystyle\sum_{j=1}^M \sum_{K\in\mathcal{T}_h}\beta_j^{(K)} (\varphi_l,p_{K} \hat{z}_j^{(K)})(\varphi_m,\hat{z}_j^{(K)}),\,
B^{S2}_{lm}=\displaystyle\sum_{j=1}^M \sum_{K\in\mathcal{T}_h}\beta_j^{(K)} (\varphi_l',p_{K} \hat{z}_j^{(K)})(\varphi_m,\hat{z}_j^{(K)}),\\ \noalign{\smallskip}
B^{S3}_{lm}=\displaystyle \sum_{j=1}^M \sum_{K\in\mathcal{T}_h}\beta_j^{(K)} (\varphi_l,p_{K} \hat{z}_j^{(K)})(\varphi_m',\hat{z}_j^{(K)}),\,
B^{S4}_{lm}=\displaystyle \sum_{j=1}^M \sum_{K\in\mathcal{T}_h}\beta_j^{(K)} (\varphi_l',p_{K} \hat{z}_j^{(K)})(\varphi_m',\hat{z}_j^{(K)}),\\ \noalign{\smallskip}
b^{S1}_{l}=\displaystyle -\gamma \sum_{j=1}^M \sum_{K\in\mathcal{T}_h}\beta_j^{(K)} (f,p_{K} \hat{z}_j^{(K)})(\varphi_l,\hat{z}_j^{(K)}),\,
b^{S2}_{l}=\displaystyle c \sum_{j=1}^M \sum_{K\in\mathcal{T}_h}\beta_j^{(K)} (f,p_{K} \hat{z}_j^{(K)})(\varphi_l',\hat{z}_j^{(K)}),
\end{array}
\end{equation}
with $\beta_j^{(K)}=1/\eta_j^{(n,K)},$ $\hat{z}_j^{(K)}=\tilde{\omega}_j^{(K)}/\|\tilde{\omega}_j^{(K)}\|_{p_{K}}$ where $\eta_j^{(K)}$ and $\tilde{\omega}_j^{(K)}$ given in expressions (\ref{sigma}) and $p_{K}=e^{-\frac{c\,x}{\mu}}.$ 

Note that, when $\gamma=0,$ $A^{S}$ coincides with the sub-grid matrix of the stationary advection-diffusion problem \cite{ChaconDia}. In the same way,  when $\gamma=0,$ $b^{S1}_{l}=0$ and $b^{S2}_{l}$ coincides with the sub-grid independent term of the stationary advection-diffusion problem \cite{ChaconDia}, which, in the case of constant velocity, is also zero. 
 
 Now we show two numerical tests, one in an advection dominated regime and another in a reaction dominated regime. First, we have considered the same values as in \cite{ChaconDia} $h=1/40, c=400, \mu=1,$ and added the reaction term $\gamma=1,$ see Fig. \ref{RCD1}(a). We observe similar behavior as in the convection-diffusion problem, as the dominant term here is the velocity  $c.$ Regarding the stabilized solution, we also observe that the solutions obtained with even number of eigenpairs present wiggles, and those with odd number do not. 
Second, we have considered a reaction dominant case. We have taken $h=1/40, c=1, \mu=1,$ and $\gamma=1000,$ see Fig. \ref{RCD1}(b). We observe that in this case, the Galerkin solution (in red) is detached from the exact one (in blue), see Fig. \ref{RCD2}. Regarding the stabilized solution, that approximates the exact solution in the grid-nodes when the number of eigenpairs tends to infinity, we observe a better performance in the case of an odd number of eigenpairs.

\begin{figure}[!ht]
\begin{center}
\begin{tabular}{ll}(a)&(b)\\
\includegraphics[width=0.5\linewidth]{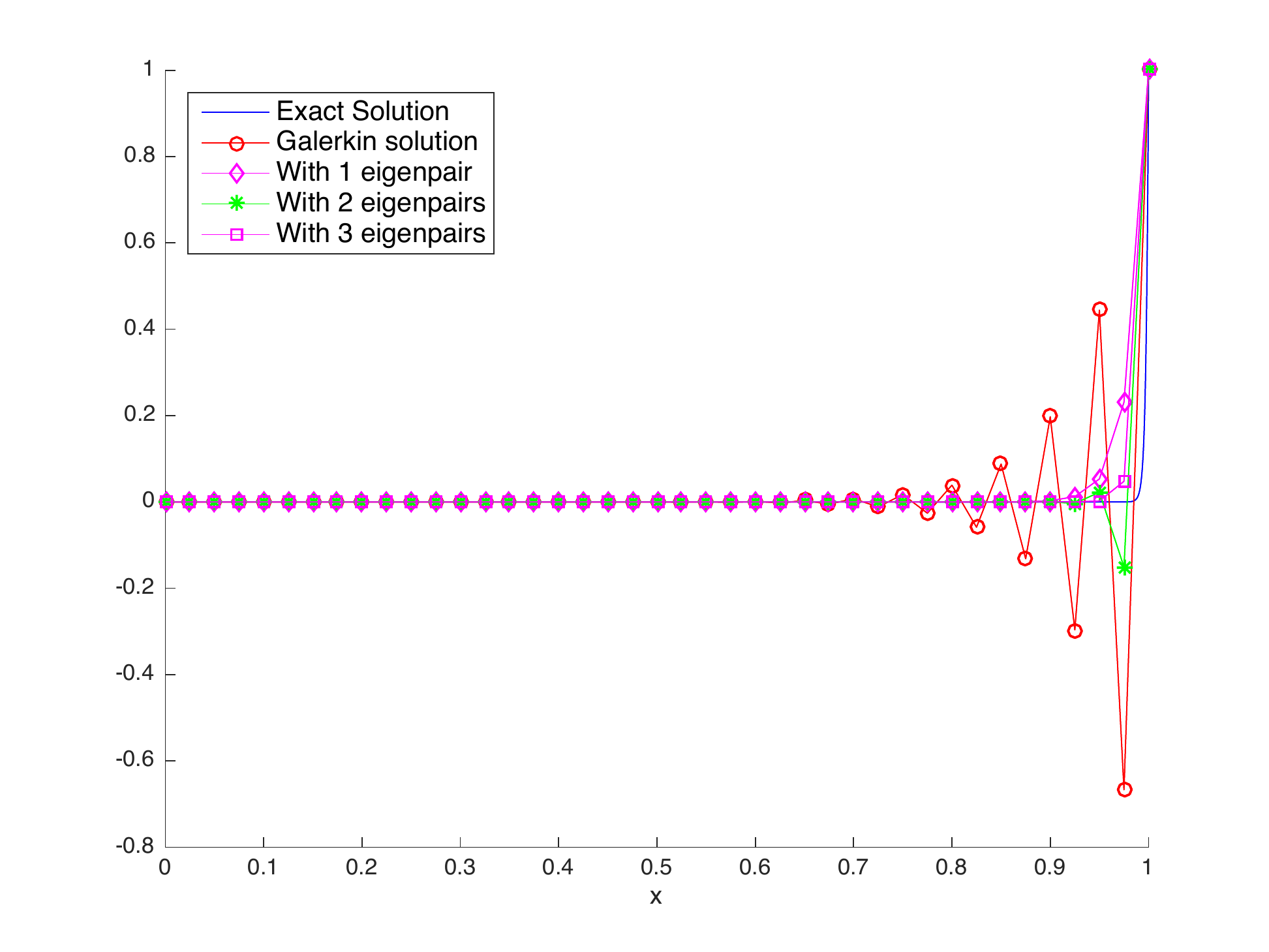}&
\includegraphics[width=0.5\linewidth]{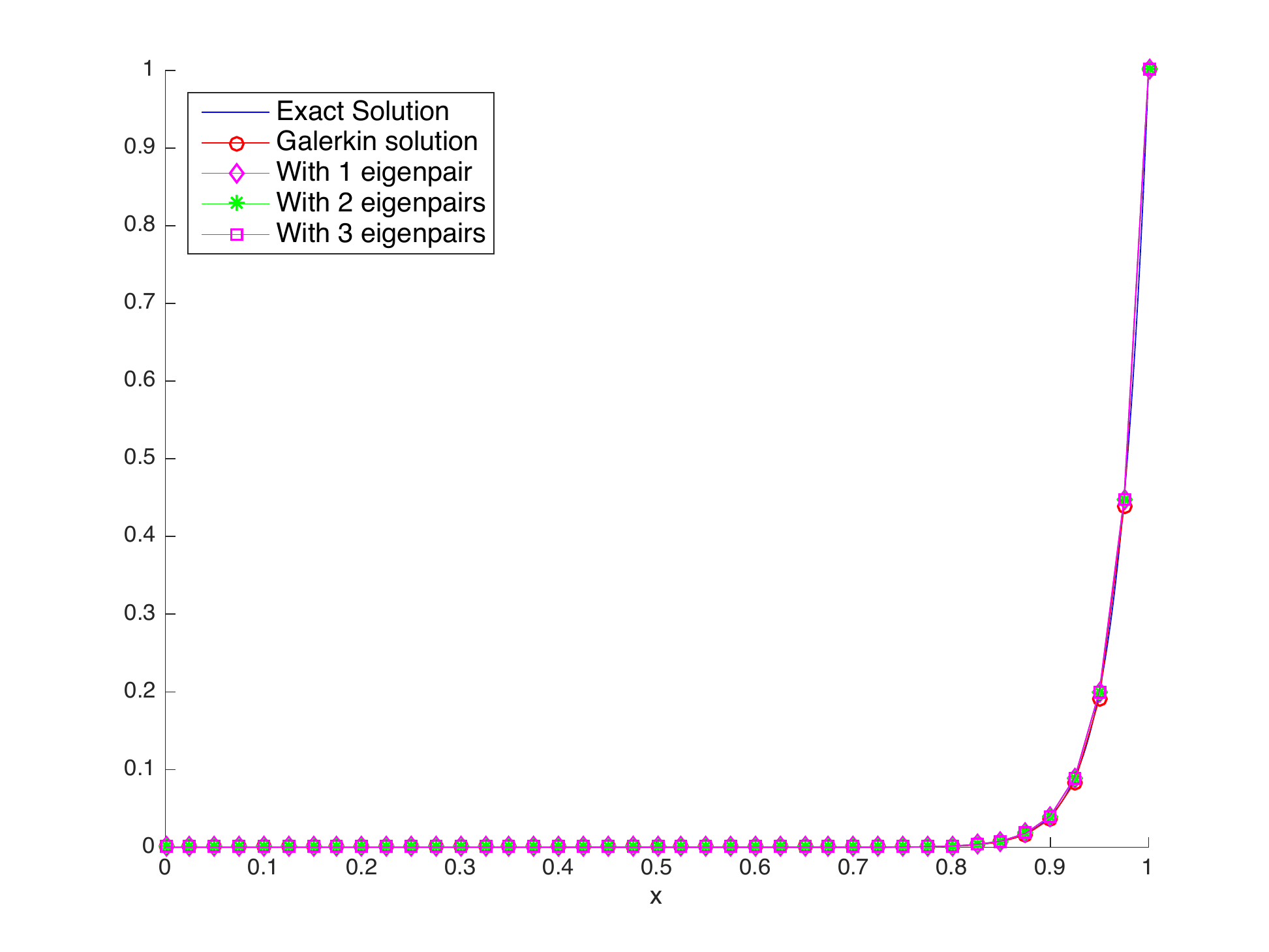}
\end{tabular}
\end{center}
\caption{\label{RCD1} Exact, Galerkin and stabilized solutions of the 1D advection-diffusion-reaction problem with (a) $\gamma=\mu=1$ and $c=400$ and (b) $\gamma=1000$ and $\mu=c=1$.}
\end{figure}

%

\begin{figure}[!ht]
\begin{center}
\begin{tabular}{ll}(a)&(b)\\
\includegraphics[width=0.5\linewidth]{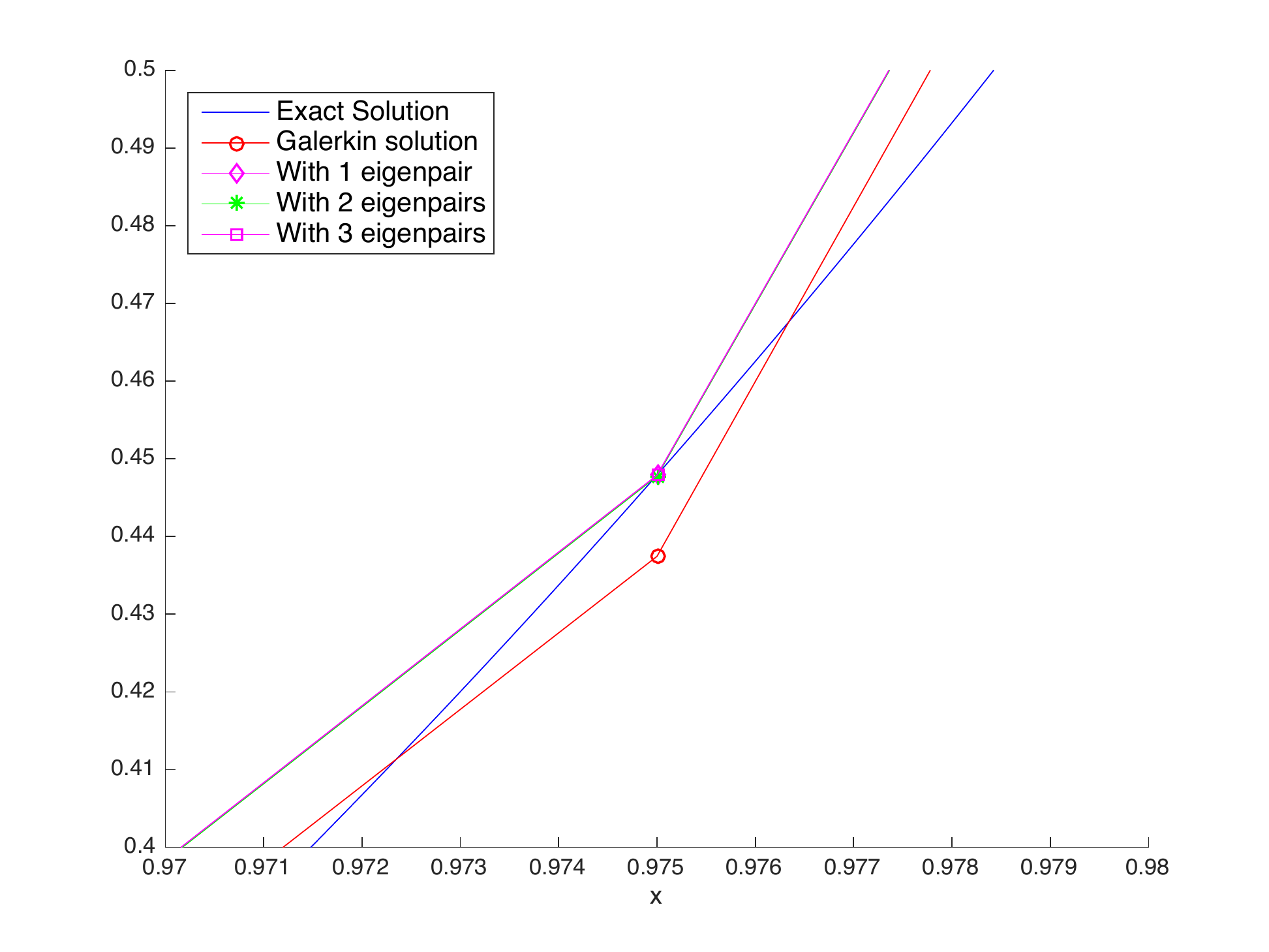}&
\includegraphics[width=0.5\linewidth]{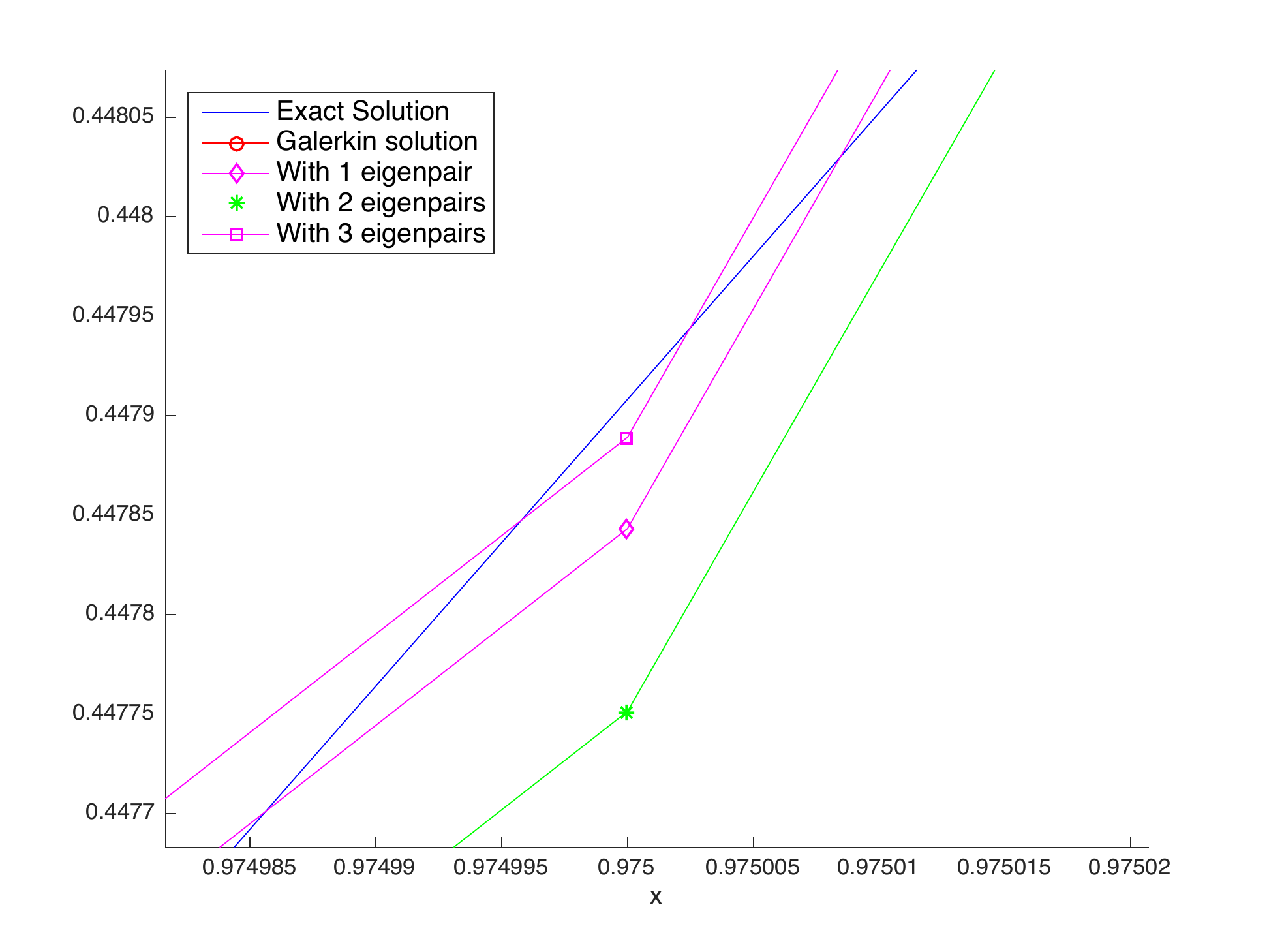}
\end{tabular}
\end{center}
\caption{\label{RCD2} Exact, Galerkin and stabilized solutions of the 1D advection-diffusion-reaction problem with $\gamma=1000$ and $\mu=c=1$ (zoom of Fig. \ref{RCD2} (b)).}
\end{figure}
Regarding the $h$-convergence order (computed when $\mu=c=\gamma=1$), we have studied the errors between the discrete solution and the piecewise interpolates of the exact solution in spaces $X_h$ and $X_{h/10}$ (a much finer mesh). We obtain order 2 in $L^2(0,1)$ and order 1 in $H^1(0,1)$ working with the fine mesh, and order 
2 in  $L^2(0,1)$ and $H^1(0,1)$ working with the coarse mesh. Thus, the well-known property that the discrete solution of the VMS formulation of steady advection-diffusion equation is exact at grid nodes is in some sense inherited by the VMS-spectral discretization of the advection-diffusion-reaction equation, but relaxed to a second order approximation, which anyhow is beyond the theoretical first order that should be inherited from interpolation.

The error behavior with respect to $M$ number of eigenpairs is illustrated in Fig. \ref{convM} for the two numerical tests previously shown in $L^\infty(0,1)$. We obtain order 3 in both cases.

\begin{figure}[!ht]
\begin{center}
\begin{tabular}{ll}(a)&(b)\\
\includegraphics[width=0.5\linewidth]{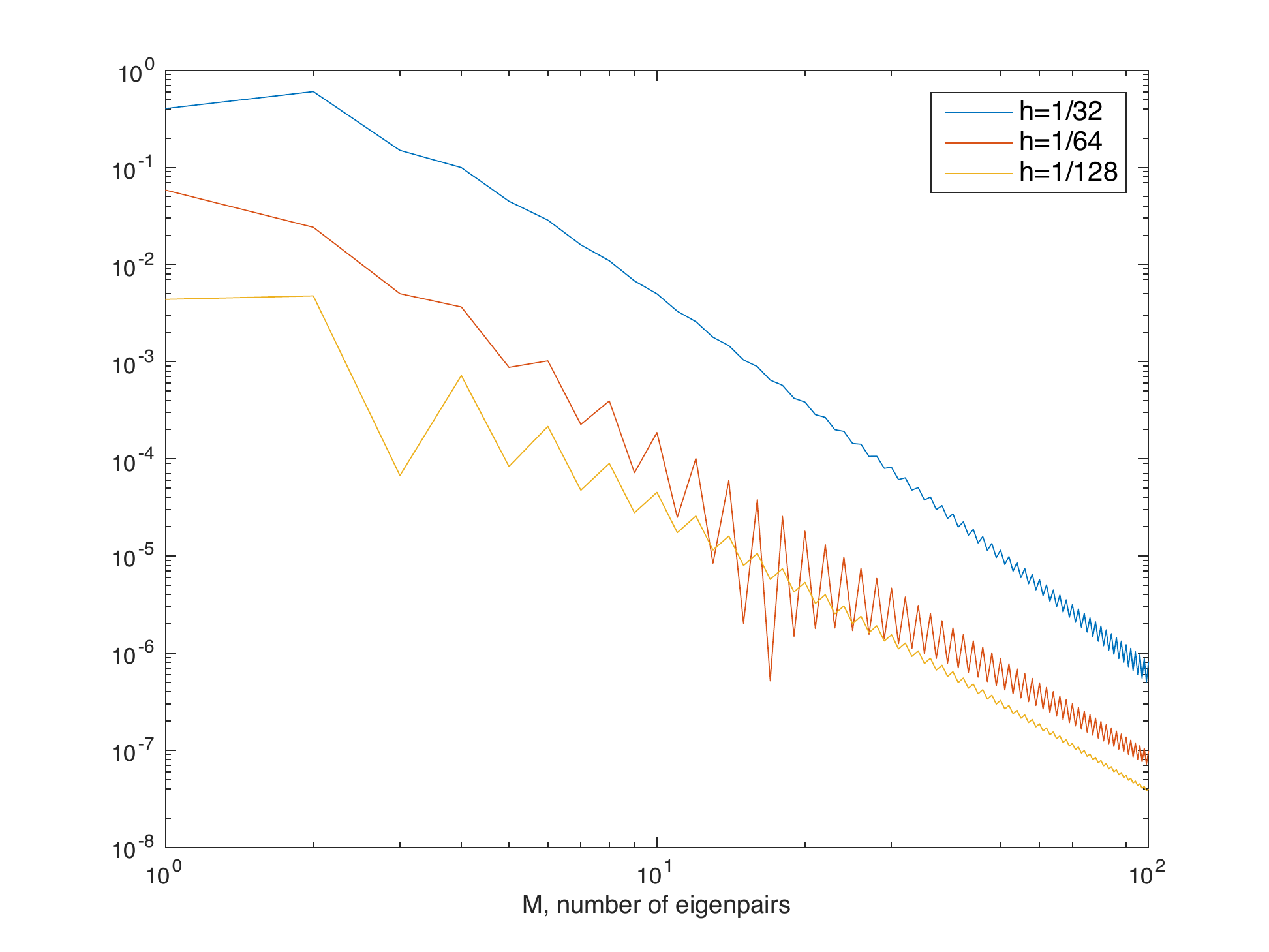}&
\includegraphics[width=0.5\linewidth]{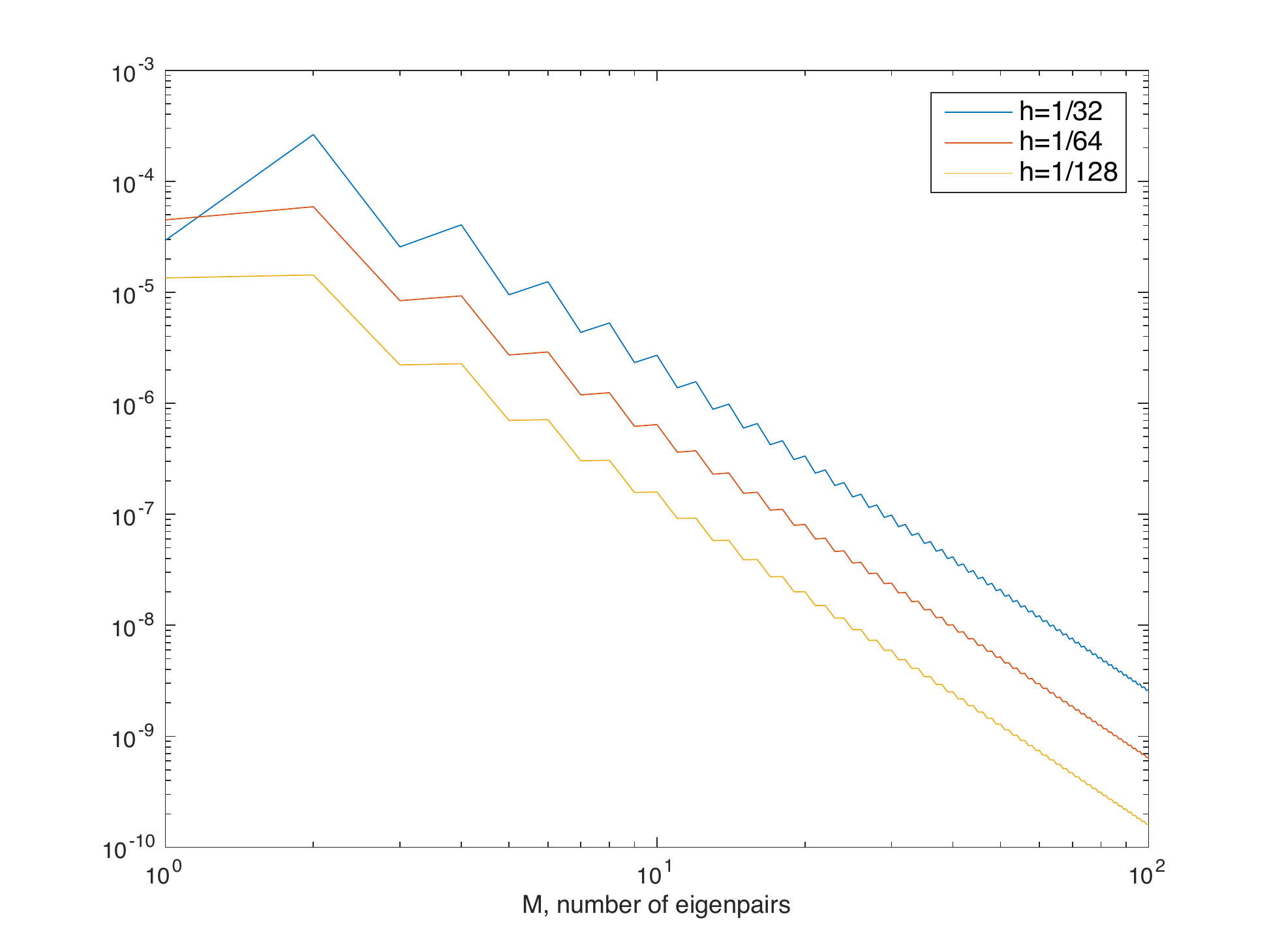}
\end{tabular}
\end{center}
\caption{\label{convM} Error in $L^\infty(0,1)$ for different values of fixed $h$ with respect to $M$ number of eigenpairs, for the 1D advection-diffusion-reaction problem with  (a) $\gamma=\mu=1$ and $c=400$ and (b) $\gamma=1000$ and $\mu=c=1$.}
\end{figure}

%

\subsection{Evolutive advection-diffusion equation}
In this section, we consider problem (\ref{EAD}) in the 1D case with constant velocity $c.$
Consider a uniform partition $\{0=t_0<t_1<...<t_N=1\}$ of the interval $[0,1],$ with time-step size $k=1/N,$ uniformly spaced nodes $0=x_1<...<x_{r+1}=1$ and $\{\varphi_{K=1}^{r+1}\}$ the piecewise affine basis functions associated to these nodes.
We look for $\displaystyle {U}_h^{n+1}=\sum_{m=2}^r U^m\varphi_m,$ with $U^m\in\mathbb{R}$ solution of the VMS discretization (\ref{ecuVMS_M}).
Thus, $\mathbf{U}^{n+1}=(U^1,...U^r)^T$  is the solution of 
\begin{equation}\label{ecuev}
\begin{array}{l}
M^{n+1}\mathbf{U}^{n+1}=\mathbf{d}^{n+1},\quad n=0,1,...,N-1,\\
\mathbf{U}^{0}=\mathbf{U}(0),
\end{array}
\end{equation}

where $M^{n+1}\in\mathbb{R}^{(r-1)\times(r-1)}$ and $\mathbf{d}^{n+1}\in\mathbb{R}^{n-1}$ are defined as
$$M^{n+1}=M^E+k (M^C+\mu M^D)+M^{S,n},\quad \mathbf{d}^{n+1}=\mathbf{d}^E+\mathbf{d}^S,$$
where $M^R,M^C,M^D$ and $M^{S,n}$ are respectively, the evolution, convection, diffusion and sub-grid matrices, $\mathbf{d}^E$ is the independent term and $\mathbf{d}^S$ is the stabilized independent term, defined by
\begin{equation*}
\begin{array}{ll}
M^E_{lm}=(\varphi_l,\varphi_m),&M^C_{lm}=(\varphi_l',\varphi_m),\quad\mbox{for}\quad l,m=1,...,n,\\ \noalign{\smallskip}
M^D_{lm}=(\varphi_l',\varphi_m'),&M^{S,n}=B^{S1,n}-c_nk\,B^{S2,n}+c_nk\,B^{S3,n}+c_n^2k^2\,B^{S4,n},\quad\mbox{for} \quad l,m=1,...,n,\\ \noalign{\smallskip}
d^E_{l}=(\mathbf{U}^{n},\varphi_l),&d_l^S=d_l^{S1}+d_l^{S2},\quad\mbox{for} \quad l,m=1,...,n,\\ \noalign{\smallskip}
\end{array}
\end{equation*}
where
\begin{equation*}
\begin{array}{l}
d^{S1}_{l}=\displaystyle - \sum_{j=1}^M \sum_{K\in\mathcal{T}_h}\beta_j^{K,t}(f^{n+1},p_{K} \hat{z}_j^{(K)})(\varphi_l,\hat{z}_j^{(K)}),\,
d^{S2}_{l}=\displaystyle c k \sum_{j=1}^M \sum_{K\in\mathcal{T}_h}\beta_j^{K,t}(f^{n+1},p_{K} \hat{z}_j^{(K)})(\varphi_l',\hat{z}_j^{(K)}),
\end{array}
\end{equation*}
where $B^{Si,n},$ $i=1,...,4$ are defined similarly to the $B^{Si}$ defined of (\ref{expB}) with obvious changes in notation. 

Now, we are going to perform some numerical tests in order to prove the reliability of the method.
First, we consider problem (\ref{EAD}) in the 1D case with constant velocity $c=1000,$ diffusion coefficient $\mu=1,$ without source term and with the initial condition

\begin{equation}\label{U01}
U_0=
\left\{\begin{array}{l}
1 \quad \mbox{if}\quad |x-0.45|\leq 0.25,\\
0 \quad \mbox{otherwise}.\\
\end{array}\right.
\end{equation}

We consider a spatial mesh with $h=1/50$ and time-step $k=10^{-3}.$ The Galerkin solution in the first five time-steps is represented in Fig.~\ref{Ev1} left panel and the spectral solution in the first five time-steps is represented in Fig.~\ref{Ev1} right panel with $M=14$ eigenpairs (in green) and $M=15$ eigenpairs (in magenta). As it can be observed, while Galerkin solution presents spurious oscillations, the spectral solution with $M=15$ eigenpairs does not, and the spectral solution with $M=14$ presents small peaks. In general, as in the stationary advection-diffusion problem \cite{ChaconDia}, we also observe that solutions with an even number of eigenpairs can present wiggles, and those with an odd number do not. 

Note that in general the matrix of the method formulated as \eqref{ecuev} is expensive to compute. A less costly formulation comes from the expression \eqref{APVMS}, that requires to pre-compute the approximate stabilized coefficients $\tau_K^M$ by \eqref{definf}. In fact, this formulation has a computational cost quite close to the standard VMS formulation, given by \eqref{EXVMS}. We here prefer use the exact formulation \eqref{ecuev} to avoid in a first step the computation of the coefficients $\tau_K^M$.

\begin{figure}[!ht]
\begin{center}
\includegraphics[width=0.7\linewidth]{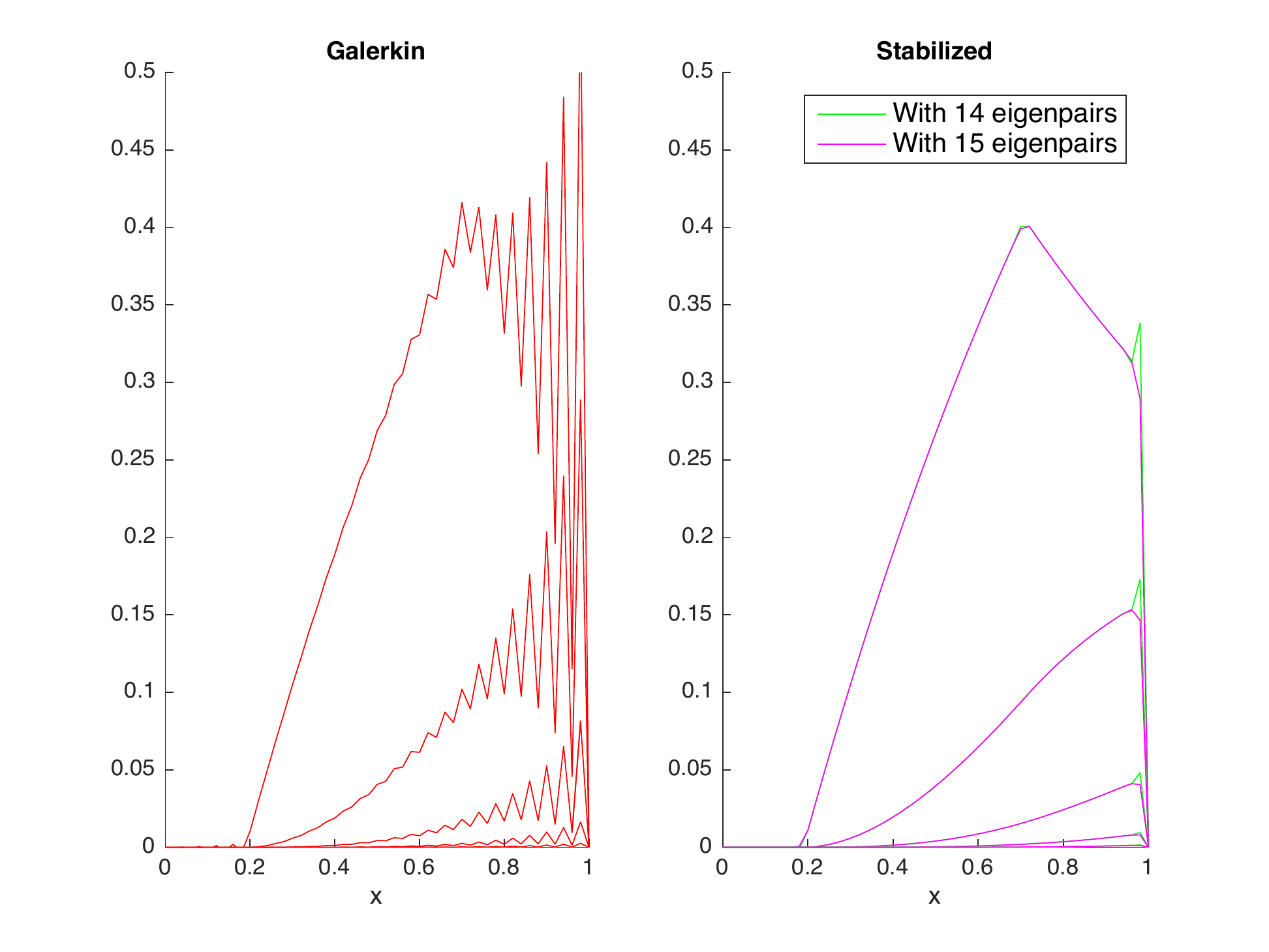}
\end{center}
\caption{\label{Ev1} 
Solution of problem  (\ref{EAD}) when $c=1000,\mu=1,f=0$ and $U_0$ given in (\ref{U01}) in the first five time-steps with $k=10^{-3}.$
Galerkin solution (left panel) and spectral solution (right panel) with $M=14$ eigenpairs (in green) and $M=15$ eigenpairs (in magenta).}
\end{figure}

\subsubsection{Coincidence at grid nodes with the exact solution in the first time-step}
In this subsection, we illustrate the fact that the spectral method approximation of the solution tends to the exact solution at grid nodes when $M\rightarrow \infty$ at the first time-step.  
In particular, we consider problem (\ref{EAD}) in the 1D case with constant velocity $c=400,$ diffusion coefficient $\mu=1,$ without source term and with the initial condition given in (\ref{U01}).

We consider a spatial mesh with $h=1/50$ and time-step $k=10^{-5}.$ The solution in the first time-step is represented in Fig.~\ref{Ev1step} (a) and a zoom around $x=0.7$ in (b). The line in blue is the best approximation of the solution, which is computed  by Galerkin method with a refined mesh $h=1/500,$ the red line corresponds with the Galerkin solution with $h=1/50$ and the green
one is the spectral solution with $M=5$ eigenpairs. As it can be seen, while Galerkin solution with the coarse mesh presents peaks, the spectral solution does not, and already with 5 eigenpairs approaches quite well the exact solution at grid nodes. However, this property is not kept in the next steps of the integration. The reason is because this fact only happens if the initial condition in the current integration step coincides with the initial condition of the exact solution, which only happens in the first step.  
\begin{figure}[!ht]
\begin{center}
\begin{tabular}{ll}(a)&(b)\\
\includegraphics[width=0.5\linewidth]{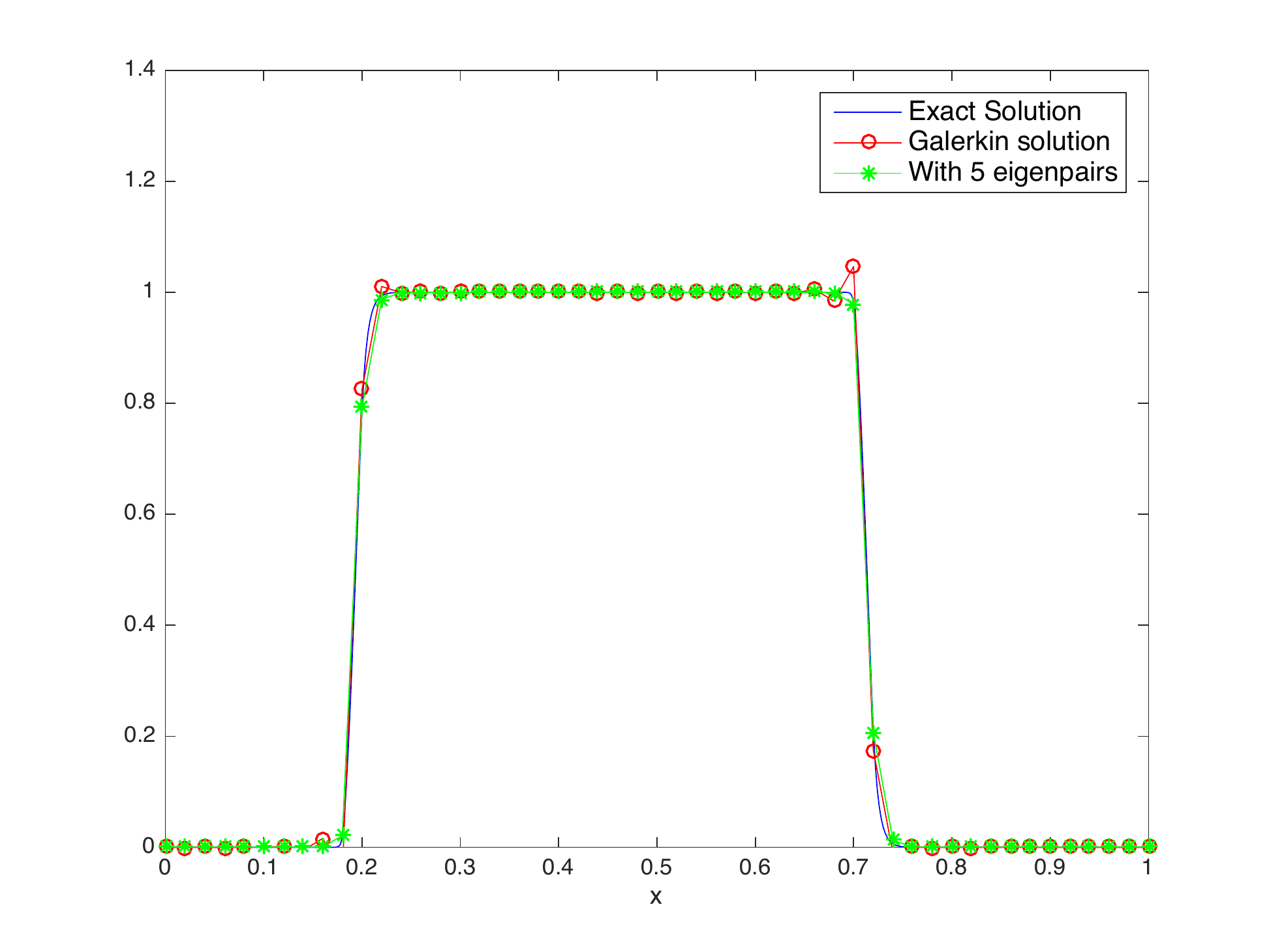}&
\includegraphics[width=0.5\linewidth]{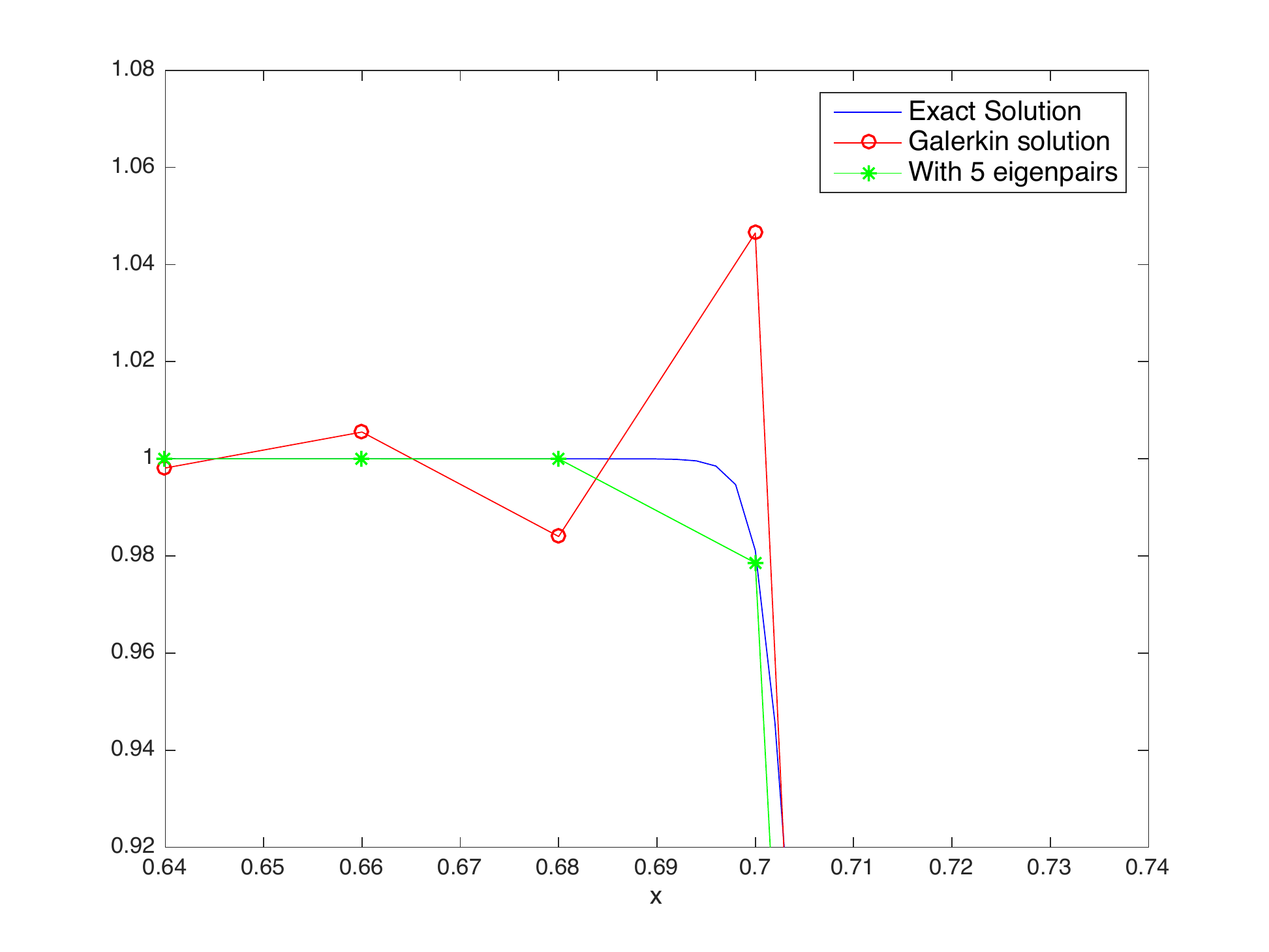}
\end{tabular}
\end{center}
\caption{\label{Ev1step}  
Solution of problem  (\ref{EAD}) when $c=400,\mu=1,f=0$ and $U_0$ given in (\ref{U01}) in the first time-step $t=10^{-5}.$ The line in blue is the best approximation of the solution, which is computed  by Galerkin method with a refined mesh $h=1/500,$ the red line corresponds with the Galerkin solution with $h=1/50$ and the green one is the spectral solution with $M=5$ eigenpairs.}
\end{figure}

\subsubsection{Convergence orders}

Regarding the $h$-convergence order (computed when $\mu=c=1$), we obtain order 2 in \break $L^\infty((0,T);L^2(0,1))$ and order 1 in $L^2((0,T);H^1(0,1))$ working with the fine mesh and order 
2 in $L^\infty((0,T);L^2(0,1))$ and $L^2((0,T);H^1(0,1))$ working with the coarse mesh, see Fig.~\ref{convMEv}. Thus, also in the evolution advection-diffusion-convection problem takes the super-convergence effect at the grid nodes.

\begin{figure}[!ht]
\begin{center}
\begin{tabular}{ll}(a)&(b)\\
\includegraphics[width=0.5\linewidth]{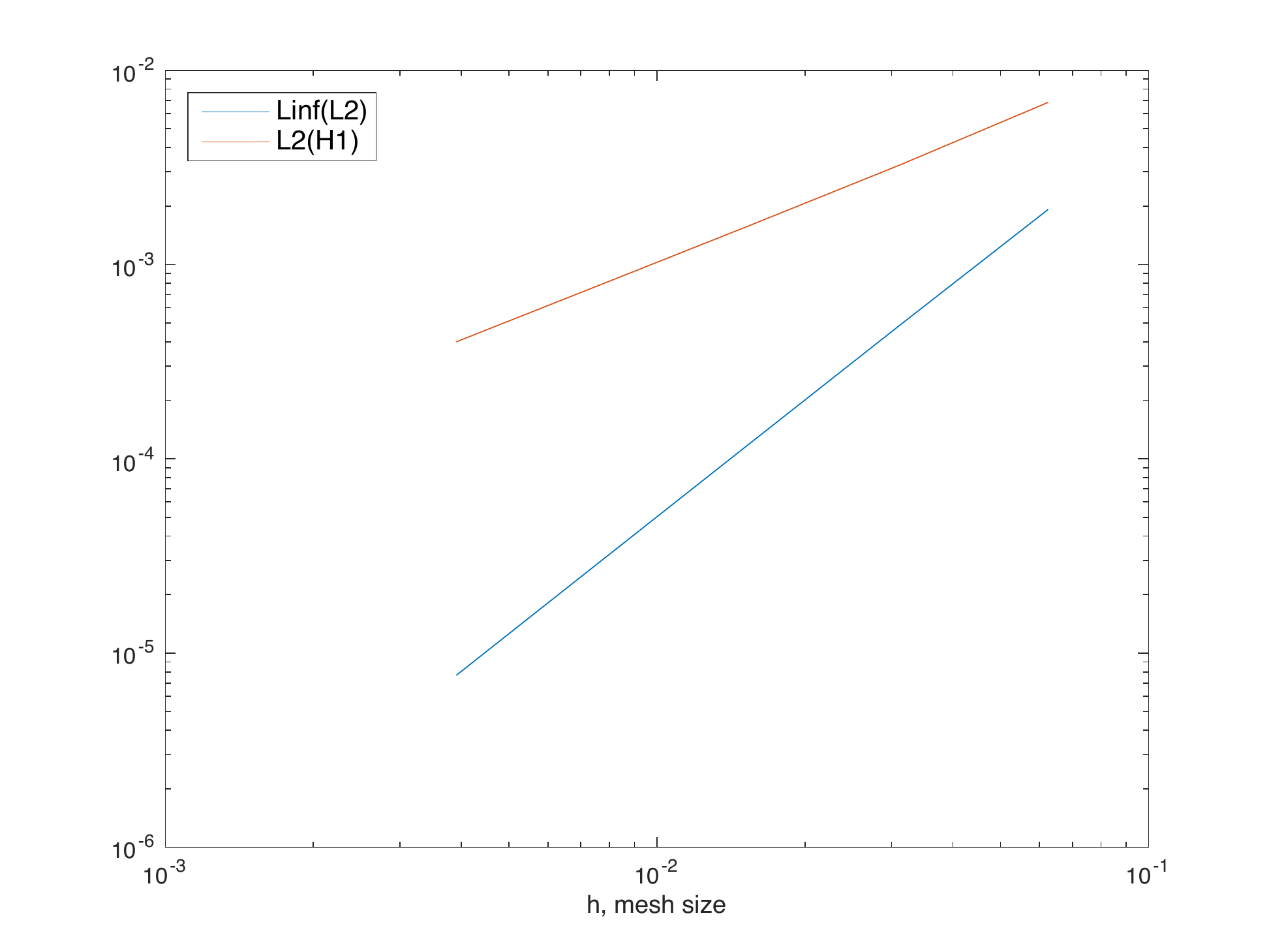}&
\includegraphics[width=0.5\linewidth]{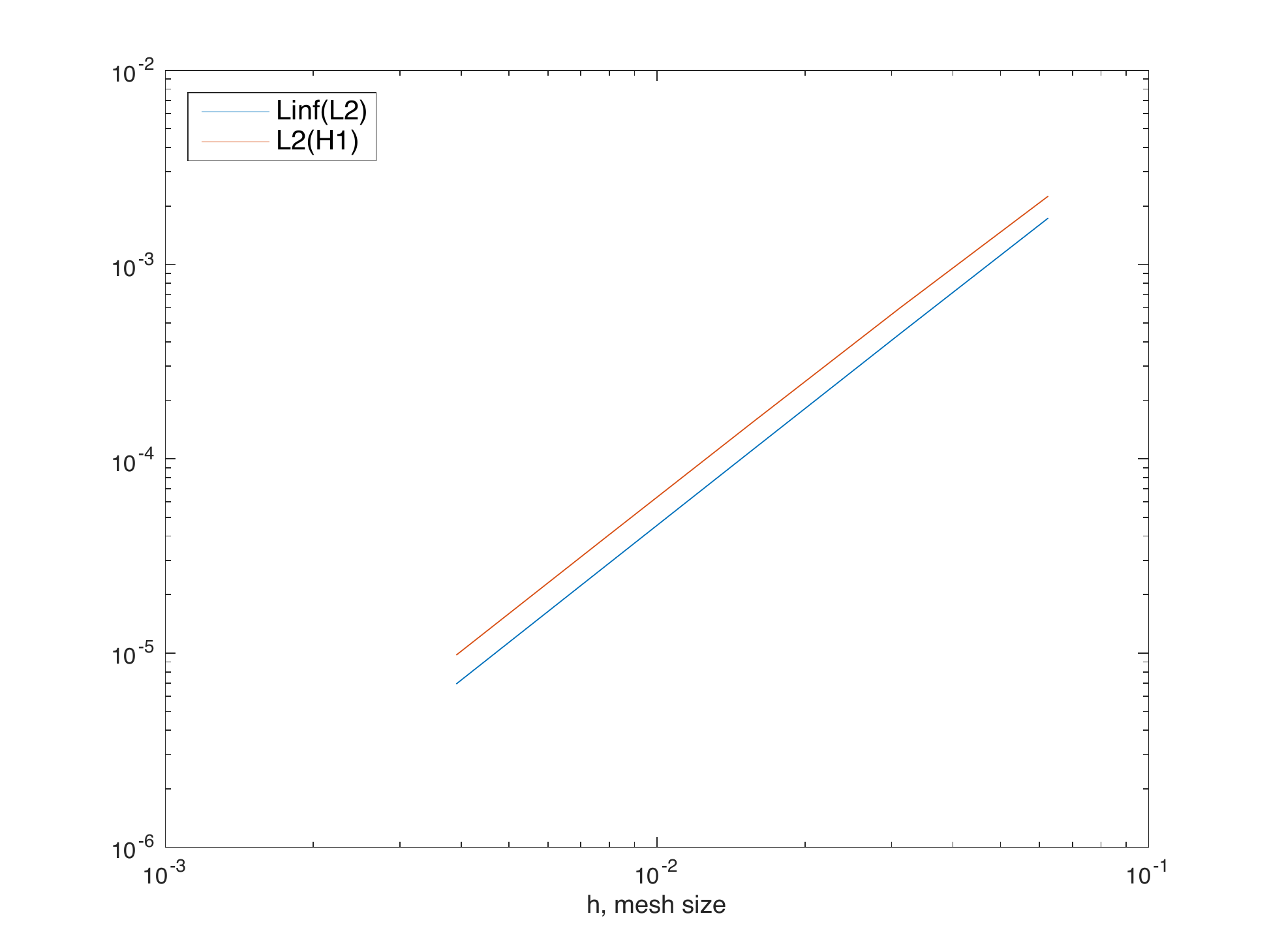}
\end{tabular}
\end{center}
\caption{\label{convMEv} Error in norm   $L^\infty((0,T);L^2(0,1))$ and $L^2((0,T);H^1(0,1))$ for fixed $M=10$ eigenpairs with respect to mesh size $h$ for the 1D  evolutive advection-diffusion problem with $c=\mu=1$ computed with (a) fine mesh and (b) coarse mesh.}
\end{figure}

Regarding the $k$-convergence order (computed when $\mu=c=1$), we obtain order 1 in $L^\infty((0,T);L^2(0,1))$ and $L^2((0,T);H^1(0,1)).$ 

The error behavior with respect to $M$ number of eigenpairs is illustrated in Fig. \ref{convM} in $L^\infty((0,T);L^2(0,1))$ and $L^2((0,T);H^1(0,1)).$ We obtain order 4 in both norms, thus a reduced number of modes is needed to obtain accurate approximations of smooth solutions.

\begin{figure}[!ht]
\begin{center}
\begin{tabular}{ll}(a)&(b)\\
\includegraphics[width=0.5\linewidth]{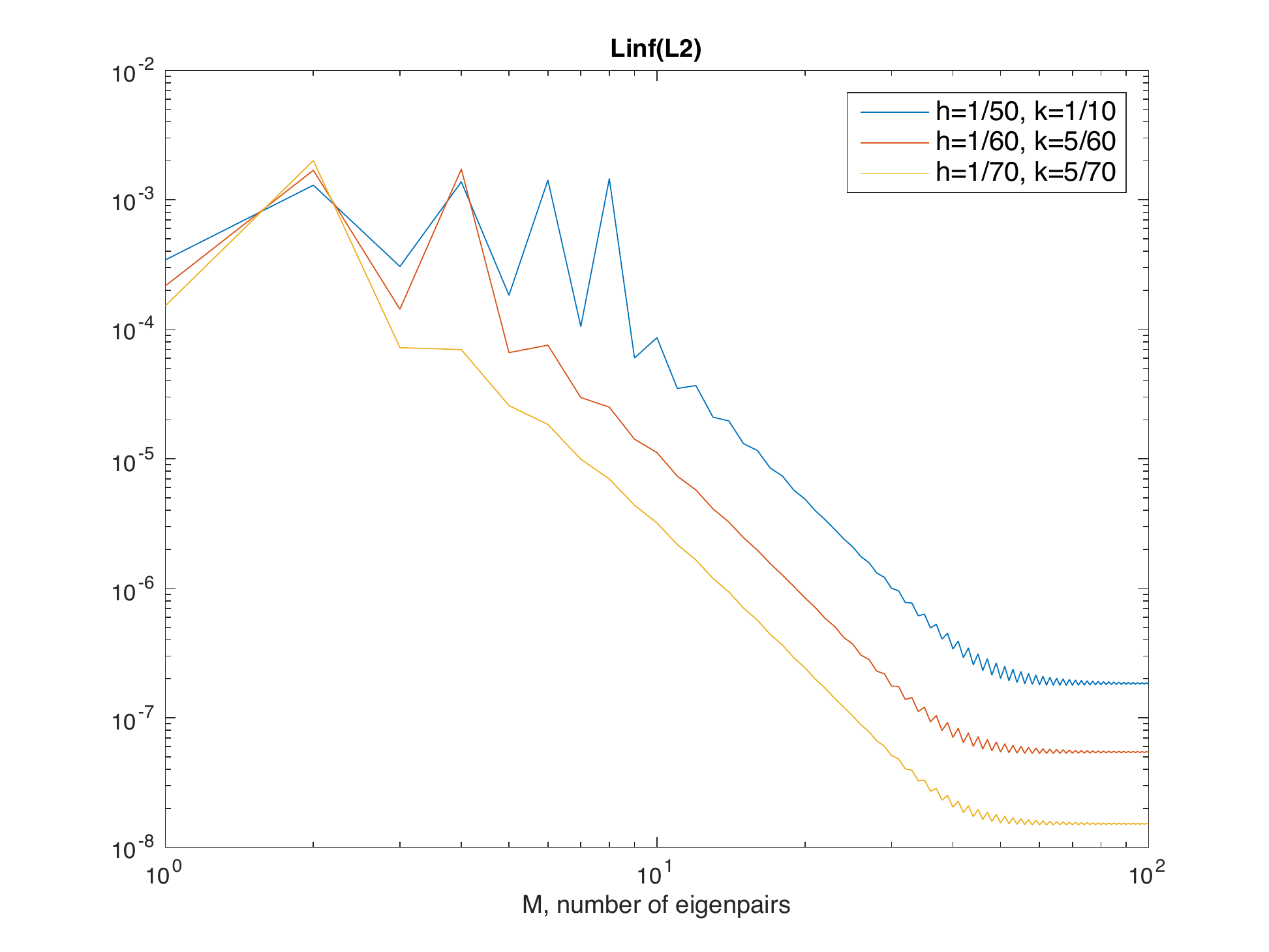}&
\includegraphics[width=0.5\linewidth]{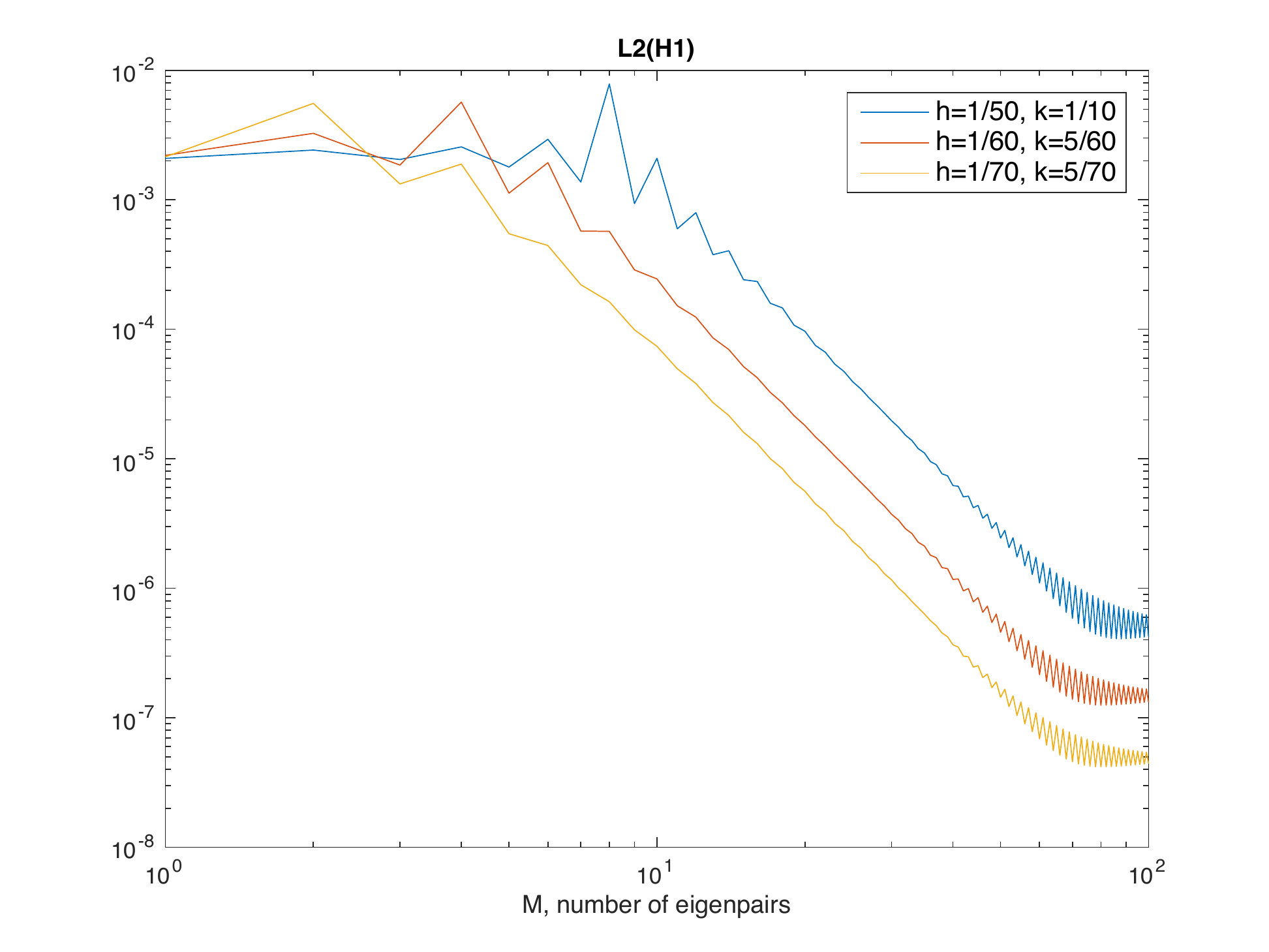}
\end{tabular}
\end{center}
\caption{\label{convM}  Error in norm  (a) $L^\infty((0,T);L^2(0,1))$ and (b) $L^2((0,T);H^1(0,1))$ for fixed $k/h=5$ with respect to $M$  eigenpairs for the 1D  evolutive advection-diffusion problem with $c=1000,$ $\mu=1$ computed with coarse mesh.}
\end{figure}



\subsubsection{Parabolic problems at small time-steps}
We consider in this subsection problems where spurious oscillations appear in the Galerkin solution due to extra small time-steps, as it was reported, for instance, in \cite{hauke}. In particular, these spurious oscillations can happen when $CFL<CFL_{bound}=P_h/(3(1-P_h)),$ (see  \cite{hauke}) being $P_h$ the element P\'eclet number $P_h=h|c|/(2\mu).$
Here, we consider the same problem as in the previous subsection but with $c=20,$ $h=1/100$ and the time-step is chosen such that $CFL/CFL_{bound}=1/2.$ We obtain the results shown in Fig.~\ref{Evhauke}, where we have represented the first five time-steps. As one can see in the figure, Galerkin solution (left panel) possesses spurious oscillations that the spectral solution  with $M=11$ eigenpairs (right panel) does not. There exists also in this case an extra diffusion effect in the spectral solution. To avoid this, it is necessary to consider more accurate methods for the semi-discretization in time.

\begin{figure}[!ht]
\begin{center}
\includegraphics[width=0.7\linewidth]{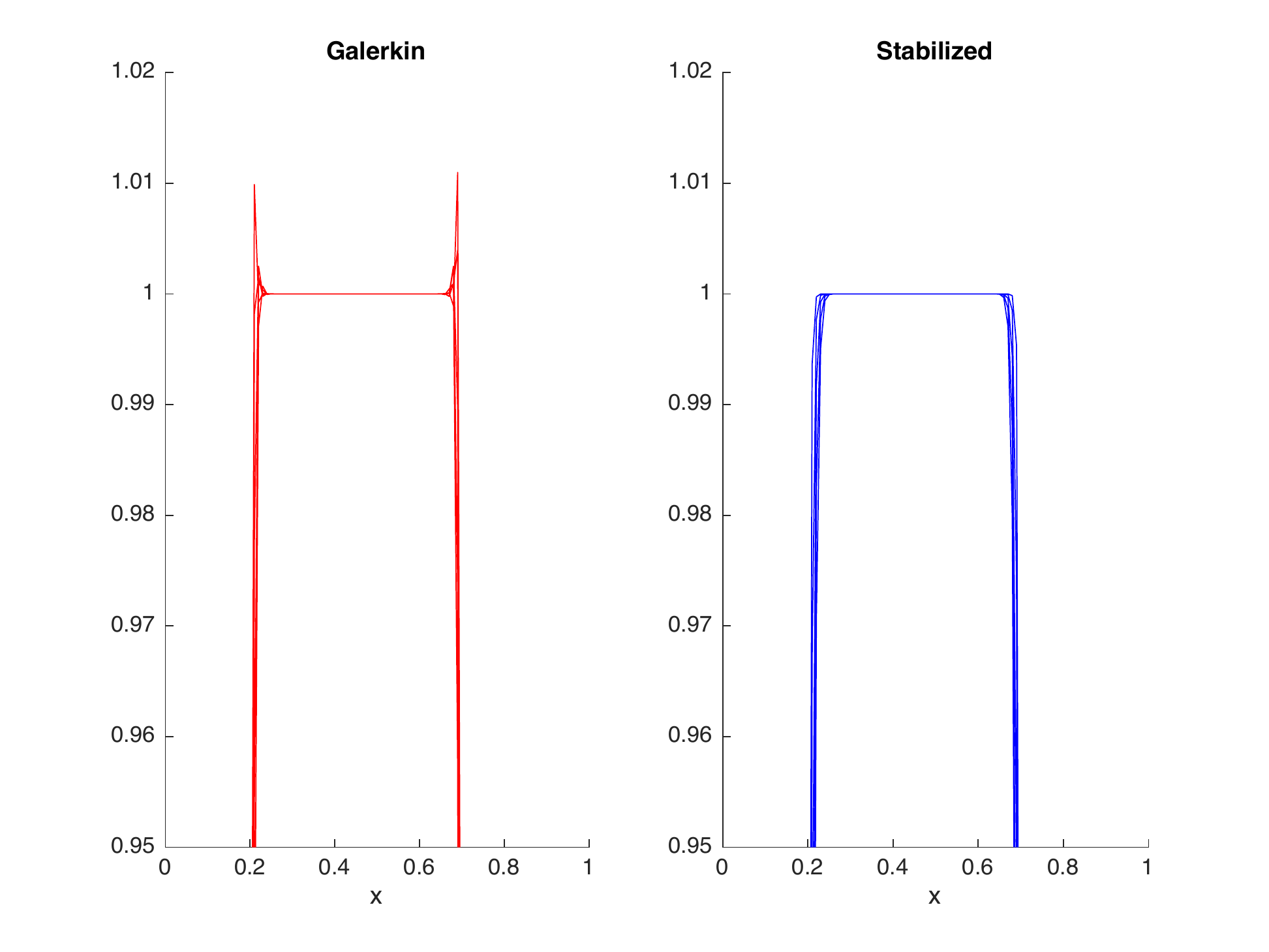}
\end{center}
\caption{\label{Evhauke} 
Solution of problem  (\ref{EAD}) when $c=20,\mu=1,f=0$ and $U_0$ given in (\ref{U01}) in the first five time-steps with $k$ such that $CFL/CFL_{bound}=1/2.$ 
Galerkin solution (left panel) and spectral solution  with $M=11$ eigenpairs (right panel).}
\end{figure}

\section{Conclusions and perspectives} \label{conclusions}
In this paper we have extended to parabolic problems the VMS-spectral method developed in \cite{ChaconDia} for elliptic problems. We have applied the method to the evolutive advection-diffusion problem. To perform the error estimates, we have distinguished between the diffusion-dominated regime and convection-dominated regime. In the second case, we have used the relation between the stabilized term expressed in terms of Green's functions and in terms of spectral functions. 

We have cast  the method as a standard VMS method with stabilized coefficients replaced by some approximated stabilized coefficients. These are computed from either the spectral eigenfunctions or from approximated element Green's functions, that in their turn are exactly computed from these eigenfunctions. Thus the computational cost of the spectral VMS method is quite close to that of VMS method, once the stabilized coefficients have been computed in an off-line step.

In looking for the solution of the evolutive advection-diffusion problem, we have naturally found the solution of the stationary advection-diffusion-reaction equation, which is also included in this work. We have performed numerical tests of both the stationary advection-diffusion-reaction equation and the evolutive advection-diffusion problem, in the 1D case. In the evolutive case, we observe that at the first time-step, the spectral method approximation of the solution tends to the exact solution at grid nodes when the number of eigenpairs tend to infinity. A natural next problem to consider, trying to maintain this property all along the integration, would be to work directly with sub-scales in time. Another path to follow, would be to consider problems in dimensions larger than one.

\appendix
\section{Appendix}\label{secapp}
In this Section, first we estimate the convergence order of $|\hat\tau_{n,K}-\hat\tau^M_{n,K}|$, and second we compute the explicit expression of $\hat\tau_{n,K}$ in the 1D case.

{\bf Estimates for $|\hat\tau-\hat\tau^M|$.} 

To simplify the notation, let us denote $\hat\tau=\hat\tau_{n,K},$ $\hat\tau^{M}=\hat\tau_{n,K}^M,$ $g_y=g_y^{(n,K)}$ and $g_y^M=g_y^{(n,K,M)}$.  Consider the solution $v \in H^1_0(K)$ of the problem
$$
\left\{\begin{array}{l}
-\Delta v = 1 \mbox{  in  } K, \\
v =0 \mbox{  on  } \partial K.
\end{array}\right.
$$
As $K$ is convex, the operator $-\Delta$ is an isomorphism from $H^{2}(K)\cap H^1_0(K)$ onto $L^{2}(K)$, and then $v \in H^2(K)$ (see Dauge \cite{dauge}). From definitions (\ref{tauk}) and (\ref{taukM}) it holds, 
\begin{equation}\label{estimtau1}
\begin{array}{rcl}
|K||\hat\tau-\hat\tau^M|&=&\displaystyle \int_{K\times K}(g_y(x)-g_y^{M}(x))dxdy=\int_{K\times K}(g_y(x)-g_y^{M}(x)) (-\Delta v)(x)\, dxdy \\ \noalign{\smallskip}
&=&\displaystyle  \int_K \frac{\langle \delta_y-\delta_y^{M}, v \rangle}{\|\Delta v\|_{L^2(K)}} \, dy.
\end{array}
\end{equation}
\par
It holds $\displaystyle v=\sum_{j\ge 1}v_j  \hat{z}_j^{(n,K)}$, with $v_j= (v,\hat{z}_j^{(n,K)})_{L^2_{p_{n,K}}(K)}$, where the series converges in ${L^2_{p_{n,K}}(K)}$. Thus,
$$
\langle \delta_y-\delta_y^M,v \rangle = \sum_{j\ge  M+1}v_j  \hat{z}_j^{(n,K)}(y),
$$
and 
\begin{eqnarray}
\int_K \langle \delta_y-\delta_y^M,v \rangle\, dy &\le& |K|^{1/2} \|p_{n,K}^{-1}\|_\infty^{1/2}\, \left\|\sum_{j\ge  M+ 1}v_j  \hat{z}_j^{(n,K)}\right\|_{L^2_{p_{n,K}}(K)} \nonumber \\
& \le& |K|^{1/2} \|p_{n,K}^{-1}\|_\infty^{1/2}\, \left (\sum_{j\ge  M+1}v_j^2\right )^{1/2}.
\label{estdeltas}
\end{eqnarray}
Observe that from Proposition \ref{propADR} the normalized eigenfunctions $\hat{z}_j^{(n,K)}$ (in $L^2_{p_{n,K}}(K)$) of the operator ${\cal L}_{n,K}$ are related to the normalized eigenfunctions $\hat{\zeta}_j^{(K)}	$ (in $L^2(K)$) of the Laplace operator on $H^1_0(K)$ by
$$
\hat{\zeta}_j^{(K)}	 = \sqrt{p_{n,K}}\, \hat{z}_j^{(n,K)},
$$
\SF{where $p_{n,K}$ is defined in expression (\ref{pesos}).}
Then $v_j=  (\sqrt{p_{n,K}}\, v,\hat{\zeta}_j^{(K)})_{L^2(K)}$, and it follows $\displaystyle \sqrt{p_{n,K}}\,v=\sum_{j\ge 1}v_j  \hat{\zeta}_j^{(K)}$, where the series converges in $L^2(K)$.  Let $\sigma_j^{(K)}$, $j=1,2,\cdots$ the eigenvalues of the Laplace operator in $H^1_0(K)$ ordered in non-decreasing values. Then,
\begin{equation}
\|\Delta(\sqrt{p_{n,K}}\, v)\|_{L^2(K)}^2=\sum_{j \ge 1}  \left |\sigma_j^{(K)}\right |^2 v_j^2.
\end{equation}
It holds
\begin{equation} \label{eq:itholdss}
\displaystyle\sum_{j\geq M+1}v_j^2\leq \frac{1}{ \left |\sigma_{M+1}^{(K)}\right |^2} \sum_{j\geq M+1} \left |\sigma_{j}^{(K)}\right |^2 v_j^2 \leq \frac{1}{ \left |\sigma_{M+1}^{(K)}\right |^2}\,\|\Delta(\sqrt{p_{n,K}}\, v)\|_{L^2(K)}^2.
\end{equation}
As 
\SF{$p_{n,K}=e^{-\frac{1}{\mu}(\mathbf{c}_n\cdot\mathbf{x})}$}
some straightforward calculations yield
\begin{eqnarray}\| \Delta (\sqrt{p_{n,K}}\, v)\|_{L^2(K)}^2\le 3 \, \|p_{n,K}\|_\infty \, \left [  \left (\frac{|{\bf c}_{n,K}|}{8 \mu}\right )^4 \,\|v\|^2_{L^2(K)} + \left (\frac{|{\bf c}_{n,K}|}{8 \mu}\right )^2 \, \|\nabla v\|^2_{L^2(K)}+ \|\Delta v\|_{L^2(K)}^2 \right ] . \label{etimas}
\end{eqnarray}
As $-\Delta$ is an isomorphism from $H^{2}(K)\cap H^1_0(K)$ on to $L^{2}(K)$, there exists two constants $C_{i,K}$, $i=1,2$ such that
$$
\|v\|_{L^2(K) }\le C_{1,K}\, \|\Delta v\|_{L^2(K)}, \,\,
\|\nabla v\|_{L^2(K) }\le C_{2,K}\, \|\Delta v\|_{L^2(K)}.
$$
Observe that $C_{1,K}$ scales as $h_K^2$ and $C_{2,K}$ scales as $h_K$, so that
\begin{equation} \label{estimass}
\| \Delta (\sqrt{p_{n,K}}\, v)\|_{L^2(K)}^2\le C \, \|p_{n,K}\|_\infty \, \left ( Pe_{n,K}^4 +  Pe_{n,K}^2+ 1 \right )\, \|\Delta v\|_{L^2(K)}^2,
\end{equation} 
where $Pe_{n,K}= \displaystyle \frac{|{\bf c}_{n,K}| \, h_K}{\mu}$ is the element P\'eclet number and $C$ is a constant depending only on the aspect ratio of the grid elements. Note also that the eigenvalues of the Laplace operator scale as $h_K^{-2}$. Then, by estimates \eqref{estimtau1}, \eqref{estdeltas}, \eqref{eq:itholdss}-\eqref{estimass} we conclude that
\begin{equation}\label{cotatau}
  |\hat\tau-\hat\tau^M| \le C\,  h_K^{2-d/2}\, \|p_{n,K}^{-1}\|_{\infty}^{1/2}\,\|p_{n,K}\|_{\infty}^{1/2}\,  ( Pe_{n,K}^2 +1) \frac{1}{\sigma_{M+1}},
\end{equation}
where the $\sigma_i$ are the eigenvalues of the Laplace operator in the reference element. 
\par
Note that by estimate \eqref{cotatau}, for a given number of eigenfunctions $M$ there will be a range of P\'eclet numbers for which the computation will be accurate. This range will increase as $M$ increases. Note also that for 1D advection-diffusion problems, $\sigma_{M+1}$ growths like $(M+1)^2$.

{\bf Explicit expression of $\hat\tau_K$ in the 1D case}. 

Consider $b$ the solution of problem (\ref{bk}) translated to the reference element, namely,
 \begin{equation}\label{bkh}
\left\{\begin{array}{ll}
\displaystyle b+\frac{k c b^{\prime}}{h}-\frac{k \mu b^{\prime \prime}}{h^2}=1&\mbox{in }[0,1],\\ \noalign{\smallskip}
b(0)=b(1)=0. &
\end{array}\right.
\end{equation}
The solution of problem (\ref{bkh}) is given by
\begin{equation*}
b(x)=\frac{-e^{L_2
   x}\left(e^{L_1}-1\right) +e^{L_1 x+L_2}-e^{L_1
   x}+e^{L_1}-e^{L_2}}{e^{L_1}-e^{L_2}},
\end{equation*}
where 
\begin{equation}\label{L1L2}
L_1=\frac{1}{2} \left(\frac{c h}{\mu }-\frac{h \sqrt{c^2 k+4 \mu
   }}{\sqrt{k} \mu }\right)\quad\mbox{and}\quad L_2=\frac{1}{2} \left(\frac{c h}{\mu }+\frac{h \sqrt{c^2 k+4 \mu
   }}{\sqrt{k} \mu }\right).
\end{equation}
Thus, from expression (\ref{tauk}),
\begin{equation}
\hat\tau=\int_0^1 b(x) dx=\frac{\left(e^{L_1}-1\right) \left(e^{L_2}-1\right)
   L_2+L_1 \left(e^{L_1} (L_2+1)-e^{L_2}
   \left(e^{L_1}+L_2-1\right)-1\right)}{L_1
   L_2 \left(e^{L_1}-e^{L_2}\right)},
\end{equation}
where $L_1$ and $L_2$ are given in expression (\ref{L1L2}).
Bearing in mind that $h\simeq k,$ we can see that,
\begin{equation}
\hat\tau=\frac{k}{12 \mu }-\frac{k^2}{120 \mu ^2}+O(k^{5/2}).
\end{equation}

\end{document}